\documentclass{amsart}
\usepackage{amsmath}
\usepackage{amssymb}
\usepackage{MnSymbol}
\usepackage{enumerate}
\usepackage{tikz-cd}
\usepackage{extarrows}

\usepackage[
        colorlinks, citecolor=darkgreen,
        backref,
        pdfauthor={Samir Siksek},
]{hyperref}

\usepackage{comment}

\newcommand{\F}{\mathbb{F}}
\newcommand{\G}{\mathbb{G}}
\newcommand{\PP}{\mathbb{P}}
\newcommand{\Q}{\mathbb{Q}}

\newcommand{\Z}{\mathbb{Z}}

\newcommand{\eps}{\varepsilon}

\newcommand{\cZ}{\mathcal{Z}}

\newcommand{\fp}{\mathfrak{p}}

\newcommand{\OO}{\mathcal{O}}

\DeclareMathOperator{\Gal}{Gal}
\DeclareMathOperator{\Hom}{Hom}

\DeclareMathOperator{\ord}{ord}
\DeclareMathOperator{\Pic}{Pic}

\renewcommand{\setminus}{-}

\newcommand{\GL}{\operatorname{GL}}

\newtheorem{thm}{Theorem}
\newtheorem*{theorem}{Theorem}
\newtheorem*{conjecture}{Conjecture}
\newtheorem{lem}[thm]{Lemma}

\newtheorem{cor}[thm]{Corollary}
\newtheorem{prop}[thm]{Proposition}

\theoremstyle{definition}

\theoremstyle{remark}

\definecolor{darkgreen}{rgb}{0,0.5,0}

\newcommand{\baseprime}{\ell}
\newcommand{\level}{n}
\newcommand{\degree}{d}
\newcommand{\realgen}{\vartheta}

\begin{document}

\title[]{
Curves with few bad primes\\ over 
cyclotomic $\Z_\ell$-extensions
}

\begin{abstract}
Let $K$ be a number field, and $S$ a finite set of non-archimedean places of $K$, and write $\OO_S^\times$ for the group of $S$-units of $K$. A famous theorem of Siegel asserts that the $S$-unit equation $\varepsilon+\delta=1$, with $\varepsilon$, $\delta \in \OO_S^\times$, has only finitely many solutions.  A famous theorem of Shafarevich asserts that there are only finitely many isomorphism classes of elliptic curves over $K$ with good reduction outside $S$.  Now instead of a number field, let $K=\Q_{\infty,\ell}$ which denotes the $\Z_\ell$-cyclotomic extension of $\Q$.  We show that the $S$-unit equation $\varepsilon+\delta=1$, with $\varepsilon$, $\delta \in \OO_S^\times$, has infinitely many solutions for $\ell \in \{2,3,5,7\}$, where $S$ consists only of the totally ramified prime above $\ell$.  Moreover, for every prime $\ell$, we construct infinitely many elliptic or hyperelliptic curves defined over $K$ with good reduction away from $2$ and $\ell$.  For certain primes $\ell$ we show that the Jacobians of these curves in fact belong to infinitely many distinct isogeny classes. 
\end{abstract}

\author{Samir Siksek}

\address{Mathematics Institute\\
    University of Warwick\\
    CV4 7AL \\
    United Kingdom}

\email{s.siksek@warwick.ac.uk}

\author{Robin Visser}

\address{Mathematics Institute\\
    University of Warwick\\
    CV4 7AL \\
    United Kingdom}
\email{Robin.Visser@warwick.ac.uk}

\date{\today}
\thanks{
Siksek is supported by the
EPSRC grant \emph{Moduli of Elliptic curves and Classical Diophantine Problems}
(EP/S031537/1). Visser is supported by an EPSRC studentship (EP/V520226/1)}

\maketitle

\section{Introduction}
Let $\ell$ be a rational
prime and $r$ a positive integer. 
Write $\Q_{r,\ell}$ for the unique degree $\ell^r$ totally real subfield
of $\cup_{n=1}^\infty \Q(\mu_n)$, where $\mu_n$ denotes
the set of $\ell^n$-th roots of $1$. We let
$\Q_{\infty,\ell}=\cup_r \Q_{r,\ell}$;
this is the $\Z_\ell$-cyclotomic extension of $\Q$,
and $\Q_{r,\ell}$ is called the $r$-th layer of $\Q_{\infty,\ell}$.
Now let $K$ be a number field, and write
$K_{\infty,\ell}=K \cdot \Q_{\infty,\ell}$
and $K_{r,\ell}=K \cdot \Q_{r,\ell}$.
To ease notation we shall sometimes write $K_\infty$
for $K_{\infty,\ell}$. We write
 $\OO_{\infty}$ (or $\OO_{\infty,\ell}$) 
for the integers in $K_\infty$
(i.e. the integral closure of $\Z$ in $K_\infty$), and
write $\OO_r$ (or $\OO_{r,\ell})$ for the integers of $K_{r,\ell}$.
Clearly $\OO_{\infty,\ell}=\cup_r \OO_{r,\ell}$.
The motivation for the present paper 
is a series of conjectures and theorems that suggest
that the arithmetic of curves (respectively abelian varieties)
over $K_\infty$ is similar to the arithmetic of curves (respectively
abelian varieties) over $K$.
One of these is the following conjecture of Mazur \cite{Mazur_1972},
which in essence  says that the Mordell--Weil theorem
 continues to hold over $K_\infty$.
\begin{conjecture}[Mazur]\label{conj:Mazur}
Let $A/K_\infty$ be an abelian variety.
	Then $A(K_\infty)$
is finitely generated.
\end{conjecture}
Another is a conjecture of Parshin and Zarhin \cite[page 91]{Zarhin_Parshin}
which is the analogue of Faltings' theorem (Mordell conjecture) over $K_\infty$.
\begin{conjecture}[Parshin and Zarhin]
Let $X/K_\infty$ be a curve of genus $\ge 2$.
Then $X(K_\infty)$ is finite.
\end{conjecture}
A third is the following theorem of Zarhin \cite[Corollary 4.2]{Zarhin_2010},
which asserts that the Tate homomorphism conjecture
(also a theorem of Faltings \cite{Faltings} over number fields)
continues to hold over $K_\infty$.
\begin{theorem}[Zarhin]
	Let $A$, $B$ be abelian varieties defined over $K_{\infty,\ell}$,
	and denote their respective $\ell$-adic Tate modules
	by $T_\ell(A)$, $T_\ell(B)$.
Then the natural embedding 
\[
	\Hom_{K_\infty}(A,B) \otimes \Z_\ell \hookrightarrow 
	\Hom_{\Gal(\overline{K_\infty}/K_\infty)}(T_\ell(A),T_\ell(B))
\]
is a bijection.
\end{theorem}
Mazur's conjecture is now known to hold for certain elliptic curves.
For example, if  $E$ is an elliptic curve defined over $\Q$ then
$E(\Q_\infty)$ is finitely generated
thanks to theorems of Kato, Ribet and Rohrlich \cite[Theorem 1.5]{GreenbergParkCity}.
From this one can deduce \cite[Theorem 1.24]{GreenbergParkCity}
that $X(\Q_\infty)$ is finite
for curves $X/\Q$ of genus $\ge 2$ equipped with a non-constant
morphism to an elliptic curve $X \rightarrow E$ defined over $\Q$.
We also note that the conjecture of Parshin and Zarhin follows
easily from Mazur's conjecture and Faltings' theorem. Indeed, 
using the Abel-Jacobi map we can deduce from
Mazur's conjecture that $X(K_{\infty})=X(K_r)$
for suitably large $r$, and we know
that $X(K_r)$ is finite by Faltings' theorem.

\bigskip

It is natural to wonder  
whether other standard conjectures
and theorems concerning the arithmetic of curves
and abelian varieties over number fields continue to hold over $K_\infty$.
The purpose of this paper is to give counterexamples
to potential generalizations of certain theorems of 
Siegel and Shafarevich to $K_\infty$.
A theorem of Siegel (e.g.\ \cite[Theorem 0.2.8]{Abramovich})
asserts that
$(\PP^1 \setminus \{0,1,\infty\})(\OO_{K,S})$
is finite for any number field $K$ and any finite
set of primes $S$. We show that the corresponding
statement over $\Q_{\infty,\ell}$ is false,
at least for $\ell=2$, $3$, $5$, $7$. 
We denote by $\upsilon_\ell$ the totally ramified prime of $\Q_{\infty,\ell}$
above $\ell$ (the precise meaning of primes in infinite extensions
of $\Q$ is clarified in Section~\ref{sec:units}).
\begin{thm}\label{thm:Siegel}
Let $\ell=2$, $3$, $5$ or $7$. 
Let
\begin{equation}\label{eqn:S}
		S \; = \; 
		\begin{cases}
			\{ \upsilon_\ell\} & \text{if $\ell=2$, $5$, $7$}\\
			\emptyset & \text{if $\ell=3$}.
		\end{cases}
\end{equation}
Let $\OO_S$ denote the $S$-integers of $\Q_{\infty,\ell}$.
Then $(\PP^1\setminus \{0,1,\infty\})(\OO_S)$
is infinite.
\end{thm}
\noindent \textbf{Remarks.}
\begin{itemize}
\item If $S=\emptyset$ then $\OO_S=\OO_{\infty}$
is the set of integers of $\Q_{\infty,\ell}$.
In \cite{FKS2} it is shown that $(\PP^1\setminus \{0,1,\infty\})(\OO_{\infty})
=\emptyset$ for $\ell \ne 3$. 
The obstruction given in \cite{FKS2}
for $\ell \ne 3$ is local in nature.
In essence, Theorem~\ref{thm:Siegel}
complements this result, showing that we can obtain
infinitely many integral or $S$-integral points in the absence
of the local obstruction. 
The proof of Theorem~\ref{thm:Siegel}
is constructive.
\item
Theorem~\ref{thm:Siegel}
strongly suggests that the conjecture of Parshin
and Zarhin does not admit a 
straightforward generalization
to the broader context
of integral points on hyperbolic curves. 
We also remark that there is a critical difference
over $K_\infty$
between complete curves $X$ of genus $\ge 2$
and $\PP^1\setminus \{0,1,\infty\}$.
For the former, the group of 
$K_\infty$-points of the Jacobian is expected to be finitely
generated by Mazur's conjecture. For the latter, the analogue
of the Jacobian is the generalized Jacobian which is
$\G_m \times \G_m$, and its group of $K_\infty$-points
is
$(\G_m\times \G_m)(K_\infty)=\OO_\infty^\times \times \OO_\infty^\times$, which is infinitely generated.
\end{itemize}
Variants of the proof of Theorem~\ref{thm:Siegel} give the following.
\begin{thm}\label{thm:Siegel2}
Let $\baseprime = 2, 3$ or $5$. Let $S=\{\upsilon_\ell\}$
and write $\OO_S$ for the $S$-integers of $\Q_{\infty,\ell}$.
Let 
\[
	k \; \in \; \begin{cases}
		\{1,2,3,4,5,6,7,8,10,12,24\} & \text{if $\baseprime = 2, 3$}, \\
		\{1, 2, 4\} & \text{if $\baseprime = 5$}.
	\end{cases}
\]
Then
$(\PP^1 \setminus \{0,k,\infty\})(\OO_S)$
is infinite.
\end{thm}

\bigskip

Shafarevich's conjecture
asserts that for a number field $K$, a dimension $n$,
a degree $d$,
and a finite set of places $S$, there are only finitely
many isomorphism classes of polarized abelian varieties defined over $K$
of dimension $n$ with degree $d$ polarization and with good reduction away from $S$.
This conjecture was proved by Shafarevich
for elliptic curves (i.e. $n=1$) and by Faltings \cite{Faltings}
in complete
generality. If we replace $K$ by $\Q_{\infty,\baseprime}$
then the Shafarevich conjecture no longer holds. For example,
consider
\[
 	E_{\varepsilon} \; : \; \varepsilon Y^2 = X^3-X
\]
where $\varepsilon \in \OO_\infty^\times$. This elliptic
curve has good reduction away from the primes above $2$.
Moreover, $E_{\varepsilon}$, $E_{\delta}$ are isomorphic
over $\Q_{\infty}$
if and only if $\varepsilon/\delta$ is a square in $\OO_\infty^\times$.
As $\OO_\infty^\times/(\OO_\infty^\times)^2$
is infinite, we deduce that there are infinitely many
isomorphism classes of elliptic curves over $\Q_\infty$
with good reduction away from the primes above $2$.
It is however natural to wonder if 
a sufficiently weakened version of the Shafarevich
conjecture continues to hold over $\Q_\infty$.
Indeed, the curves $E_{\varepsilon}$ in the above construction
form a single $\overline{\Q}$-isomorphism class.
This it is natural to ask if, for  suitable $\ell$ and 
finite set of primes $S$, does the set of elliptic curves
over $\Q_{\infty}$ with good reduction outside $S$
form infinitely many $\overline{\Q}$-isomorphism classes?
\begin{thm}\label{thm:Shafarevich}
Let $\ell=2$, $3$, $5$, or $7$. 
	Let $S$ be given by \eqref{eqn:S} 
	and let $S^\prime=S \cup \{\upsilon_2\}$
where $\upsilon_2$ is the unique prime of $\Q_{\infty, \baseprime}$
above $2$.
	Then, there are infinitely many $\overline{\Q}$-isomorphism classes of
	elliptic curves defined over $\Q_{\infty, \baseprime}$
	with good reduction away from $S^\prime$ and with full 2-torsion in $\Q_{\infty, \baseprime}$.
	Moreover, these elliptic curves form
	infinitely many distinct $\Q_{\infty,\ell}$-isogeny classes.
\end{thm}

\subsection*{Remarks}
\begin{itemize}
\item By \cite[Lemma 2.1]{FKS2},
	a rational prime $p \ne \ell$ is inert in $\Q_{\infty,\ell}$
	if and only if $p^{\ell-1} \not \equiv 1 \pmod{\ell^2}$.
It follows from this that $2$ is inert in $\Q_{\infty,\ell}$
for $\ell=3$, $5$, $7$ and $11$.
\item Faltings' proof \cite{Faltings} of the Mordell conjecture can
be considered to have three major steps. In the first step, Faltings proves
the Tate homomorphism conjecture. In the second step,
Faltings derives the Shafarevich conjecture from
the Tate homomorphism conjecture, and in the final 
step Faltings uses the \lq Parshin trick\rq\ to 
deduce the Mordell conjecture from the Shafarevich
conjecture. Although Zarhin has extended the
Tate homomorphism conjecture to $K_\infty$,
Theorem~\ref{thm:Shafarevich} 
suggests that there is no plausible strategy for
proving the conjecture
of Parshin and Zarhin by mimicking
Faltings' proof of the Mordell conjecture.
\end{itemize}

\bigskip

It is natural to wonder if the isogeny classes appearing
in the proof of Theorem~\ref{thm:Shafarevich} are finite or infinite.
Rather reassuringly they turn out to be finite.
\begin{thm}\label{thm:finite}
Let $E$ be an elliptic
curve over $\Q_{\infty,\ell}$
without potential complex multiplication.
Then the $\Q_{\infty,\ell}$-isogeny class
of $E$ is finite.
\end{thm}

\bigskip

The original version of Shafarevich's conjecture \cite{Shafarevich},
(also proved by Faltings \cite[Korollar 1]{Faltings})
states that for a given number field $K$,
a genus $g$ and a finite set of places $S$,
there are only finitely many isomorphism classes
of genus $g$ curves $C/K$ with good reduction away from $S$.
Again this statement becomes false if we replace $K$ by 
$\Q_{\infty,\ell}$, for any prime $\ell$.

\begin{thm}\label{thm:Shafarevich2}
Let $g \ge 2$ and let $\ell = 3$, $5$, $7$, $11$ or $13$. There are infinitely many
$\overline{\Q}$-isomorphism classes
of genus $g$ hyperelliptic curves
over $\Q_{\infty, \ell}$
with good reduction away from $\{ \upsilon_2, \upsilon_\ell \}$. 
\end{thm}
\begin{thm}\label{thm:Shafarevich3}
Let $\baseprime \geq 11$  be an odd prime and let $g = \lfloor \frac{\baseprime - 3}{4} \rfloor$. There are infinitely many
$\overline{\Q}$-isomorphism classes
of genus $g$ hyperelliptic curves
over $\Q_{\infty,\baseprime}$
with good reduction away from $\{\upsilon_2,\upsilon_\baseprime\}$. 
Moreover, if 
\[
	\ell \in \{11, 23, 59, 107, 167, 263, 347, 359 \},
\]
then the Jacobians of these
curves form infinitely many distinct $\Q_{\infty, \baseprime}$-isogeny classes.
\end{thm}

\bigskip

The paper is structured as follows. In Section \ref{sec:units}
we recall basic results on units and $S$-units
of the cyclotomic field $\Q(\zeta_{\ell^n})$.
In Sections \ref{sec:sunit+}--\ref{sec:Q7}
we employ identities between cyclotomic polynomials
to give constructive proofs of Theorems~\ref{thm:Siegel} and~\ref{thm:Siegel2}.
Section~\ref{sec:elliptic} gives a proof of Theorem~\ref{thm:finite},
making use of a deep theorem of Kato to control the $\Q_{\infty,\ell}$-points
on certain modular curves. Section~\ref{sec:antishafarevich}
uses the integral and $S$-integral points on $\PP^1 \setminus \{0,1,\infty\}$
furnished by Theorem~\ref{thm:Siegel} to construct infinite
families of elliptic curves over $\Q_{\infty,\ell}$ for $\ell=2$, $3$, $5$, $7$,
with good reduction away from $\{\upsilon_2,\upsilon_\ell\}$, 
which are used to give a proof of Theorem~\ref{thm:Shafarevich}.
Sections~\ref{sec:hyp} and~\ref{sec:isog} give
proofs of Theorems~\ref{thm:Shafarevich2} and~\ref{thm:Shafarevich3},
making use of the relation, due to Kummer, between the class number of
$\Q(\zeta_{\ell^n})^+$, and the index of cyclotomic units
in the full group of units.

\bigskip

We are grateful to Minhyong Kim for drawing our attention
to the conjecture of Parshin and Zarhin, and to 
Alain Kraus and David Loeffler for useful discussions.

\section{Units and $S$-units of $\Q(\zeta)$}\label{sec:units}
Let $K$ be a subfield of $\overline{\Q}$.
We denote the integers of $K$ (i.e.\ the integral closure
of $\Z$ in $K$) by $\OO(K)$.
Let $p$ be a rational prime.
By a \textbf{prime of $K$ above $p$} we mean 
a map $\upsilon : K \rightarrow \Q \cup \{\infty\}$
satisfying the following
\begin{itemize}
\item $\upsilon(p)=1$, $\upsilon(0)=\infty$;
\item $\upsilon \vert_{K^\times} : K^\times \rightarrow \Q$ is a homomorphism;
\item $\upsilon(1+b)=0$ whenever $\upsilon(b)>0$.
\end{itemize}
Suppose $K=\cup K_n$ where $K_0 \subset K_1 \subset K_2 \subset \cdots$
is a tower of number fields (i.e.\ finite extensions of $\Q$), with $K_0=\Q$.
One sees that the primes of $K$ above $p$
are in $1$--$1$ correspondence with sequences 
$\{\fp_n\}$ where 
\begin{itemize}
	\item $\fp_n$ is a prime ideal of $\OO(K_n)$;
	\item $\fp_{n+1} \mid \fp_n \OO(K_{n+1})$;
\item $\fp_0=p \Z$.
\end{itemize}
Indeed, from $\upsilon$ one obtains the corresponding
sequence $\{ \fp_n\}$ via
the formula $\fp_n=\{\alpha \in \OO(K_n) \; : \; \upsilon(\alpha)>0\}$.
Given a sequence $\{ \fp_n\}$, we can recover 
the corresponding $\upsilon$ by letting 
\[
\upsilon(\alpha)=\ord_{\fp_n}(\alpha)/\ord_{\fp_n}(p)
\]
whenever $\alpha \in K_n^\times$.
Given a finite set of primes $S$ of $K$, we define the $S$-integers of $K$ to be
the set $\OO(K,S)$ of all $\alpha \in K$ such that $\upsilon(\alpha) \ge 0$
for every prime $\upsilon \notin S$. 
We let $\OO(K,S)^\times$ be the unit group of $\OO(K,S)$;
this is precisely the set of $\alpha \in K^\times$
such that $\upsilon(\alpha)=0$
for every prime $\upsilon \notin S$.
If $S=\emptyset$ then
$\OO(K,S)=\OO(K)$ are the integers of $K$
and $\OO(K,S)^\times=\OO(K)^\times$ are the units
of $K$.

\bigskip

Fix a rational prime $\ell$. For a positive integer $n$,
let
 $\zeta_{\ell^n}$ denote a primitive $\ell^n$-th root of $1$ which
is chosen so that 
\[
\zeta_{\ell^{n+1}}^\ell=\zeta_{\ell^n}.
\]
Let $\Omega_{n,\ell}=\Q(\zeta_{\ell^n})$; this has
degree $\varphi(\ell^n)$ where $\varphi$ is Euler totient function.
Let
\[
	\Omega_{\infty,\ell}=\bigcup_{n=1}^\infty \Omega_{n,\ell}.
\]
The prime $\ell$ is totally ramified in 
each $\Omega_{n,\ell}$, and we denote by $\lambda_n$
the unique prime ideal of $\OO({\Omega_{n,\ell}})$ above $\baseprime$.
Thus
\begin{equation}\label{eqn:ramify}
	\ell \cdot \OO({\Omega_{n,\ell}}) \, =\, \lambda_n^{\varphi(\ell^n)}.
\end{equation}
We write $\upsilon_\ell$ for the unique prime of $\Omega_{\infty,\ell}$
above $\ell$. 
For now fix $n \ge 1$
if $\ell \ne 2$ and $n \ge 2$ if $\ell=2$. 
We recall that
 $\lambda_n=(1-\zeta_{\ell^n}) \cdot \OO({\Omega_{n,\ell}})$.
If $\ell \nmid s$ then 
$(1-\zeta_{\ell^n}^s) \cdot \OO({\Omega_{n,\ell}})=\lambda_n$;
we can see this by applying the automorphism $\zeta_{\ell^n} \mapsto \zeta_{\ell^n}^s$
to \eqref{eqn:ramify}.
\begin{lem}\label{lem:cycval1}
	Let $s$ be an integer and let $t=\ord_{\ell}(s)$.
	Suppose $t<n$. Then
\[
	(1-\zeta_{\ell^n}^s) \cdot \OO({\Omega_{\level,\ell}}) \; =\; \lambda_n^{\ell^t}.
\]
Moreover,
\[
\upsilon_\ell(1-\zeta_{\ell^n}^s)=
\frac{1}{\ell^{n-1-t} (\ell-1)}.
\]
\end{lem}
\begin{proof}
	Write $\zeta=\zeta_{\ell^n}$.
	Note that $\zeta^s$
	is a primitive $\ell^{n-t}$-th root of $1$.
	Thus 
	\[
		(1-\zeta^s) \cdot \OO(\Omega_{n-t, \baseprime}) \; =\; \lambda_{n-t}.
	\]
	As $\ell$ is totally ramified
	in $\Omega_{n,\ell}$, we have 
	\[
		(1-\zeta^s)  \cdot \OO(\Omega_{n,\ell}) \;
	=\; \lambda_n^{[\Omega_{n,\ell} \, :\, \Omega_{n-t,\ell}]} \; =\;
	\lambda_n^{\ell^t}.
	\]
For the final part of the lemma,
\[
\upsilon_\ell(1-\zeta^s)=\frac{\ord_{\lambda_n}(1-\zeta^s)}{\ord_{\lambda_n}(\ell)}=\frac{\ell^t}{\varphi(\ell^n)}=
\frac{1}{\ell^{n-1-t} (\ell-1)}.
\]
\end{proof}

\subsection*{Cyclotomic units and $S$-units}
Write $V_n$ for the subgroup of $\OO(\Omega_n,\{\upsilon_\ell\})^\times$
generated by
\[
	\left\{\pm \zeta_{\ell^n}, \quad 1-\zeta_{\ell^n}^k \; : \; 
	1 \le k < \ell^n
	\right\}
\]
and let
\[ 
C_n=V_n \cap \OO(\Omega_n)^\times.
\]
The group $C_n$ is called \cite[Chapter 8]{Washington} the group of \textbf{cyclotomic units} in $\Omega_n$. We will often find it more convenient
to work with group $V_n$.
\begin{lem}\label{lem:Vnbasis}
The abelian group
$V_n/\langle \pm \zeta_{\ell}^n \rangle$
is free with basis
\begin{equation}\label{eqn:Vnbasis}
	\left\{ 1-\zeta_{\ell^n}^k \; : \; 1 \le k < \ell^n/2,
	\quad \ell \nmid k \right\}.
\end{equation}
\end{lem}
\begin{proof}
	The torsion subgroup of $V_n$ is the torsion
	subgroup of $\Omega_n^\times$ which is $\langle \pm \zeta_{\ell^n}
	\rangle$. Thus $V_n/\langle \pm \zeta_{\ell^n} \rangle$
	is torsion free.
By definition of $V_n$, the group $V_n/\langle \pm \zeta_{\ell^n} \rangle$
	is generated by $1-\zeta_{\ell^n}^k$ with $\ell^n \nmid k$.
	Write $k=\ell^r d$ with $\ell \nmid d$; thus
	$r<n$. Suppose $r \ge 1$. Then,
\[
	\begin{split}
		1-\zeta_{\ell^n}^k & =1-\zeta_{\ell^n}^{\ell^r d} \\
		 & = \prod_{i=0}^{\ell^r-1} (1-\zeta_{\ell^n}^d \zeta_{\ell^r}^i)
		\qquad \text{using} \quad 1-X^{\ell^r}=\prod_{i=0}^{\ell^r-1} (1-\zeta_{\ell^r}^i X)\\
		& = \prod_{i=0}^{\ell-1} (1-\zeta_{\ell^n}^{d+i \ell^{n-r}} ).
	\end{split}
\]
It follows that $V_n/\langle \pm \zeta_{\ell^n}\rangle$ is generated
by $1-\zeta_{\ell^n}^k$ with $\ell \nmid k$.
If $\ell^n/2 < k < \ell^n$ and $\ell \nmid k$ then
\begin{equation}\label{eqn:residueexp}
	1-\zeta_{\ell^n}^k = -\zeta_{\ell^n}^{k} (1-\zeta_{\ell^n}^{\ell^n-k}).
\end{equation}
Thus \eqref{eqn:Vnbasis}
certainly generates $V_n/\langle \pm \zeta_{\ell}^n \rangle$.
Note that \eqref{eqn:Vnbasis} has cardinality
$\varphi(\ell^n)/2$ where $\varphi$ is the Euler totient
function. It therefore suffices to show that $V_n$
has rank $\varphi(\ell^n)/2$.
A well-known theorem \cite[Theorem 8.3]{Washington} states
that $C_n$ has finite index in $\OO(\Omega_n)^\times$
and thus, by Dirichlet's unit theorem,
$C_n$ has rank $-1+\varphi(\ell^n)/2$.
We note that $C_n$ is the kernel of the surjective
homomorphism $V_n \rightarrow \Z$, sending $\mu$
to $\ord_{\lambda_n}(\mu)$. Thus $V_n$ has rank $\varphi(\ell^n)/2$
completing the proof.
\end{proof}

\begin{lem}\label{lem:quo}
Let $n \ge 2$ if $\ell \ne 2$ and $n \ge 3$ if $\ell=2$. 
	Then $V_{n-1} \subset V_n$. Moreover,
	\[
		\prod_{\substack{1 \le k <\ell^n/2\\ \ell \nmid k}} (1-\zeta_{\ell^n}^k)^{c_k}  \; \in \; \langle \pm \zeta_{\ell^n}, \, V_{n-1} \rangle
	\]
	if and only if $c_k=c_m$ whenever $k \equiv m \pmod{\ell^{n-1}}$.
\end{lem}
\begin{proof}
The group $V_{n-1}$ is generated, modulo roots of unity,
by $1-\zeta_{\ell^{n-1}}^d$ with $\ell \nmid d$. By the
	proof of Lemma~\ref{lem:Vnbasis},
\[
	\begin{split}
		1-\zeta_{\ell^{n-1}}^d  = 
		1-\zeta_{\ell^n}^{\ell d} 
		 = \prod_{i=0}^{\ell-1} (1-\zeta_{\ell^n}^{d+i \ell^{n-1}} ).
	\end{split}
\]
	The lemma follows from Lemma~\ref{lem:Vnbasis}.
\end{proof}

Given $a \in \Z_\ell$, it makes sense to reduce $a$ modulo $\ell^n$
and therefore it makes sense to write $\zeta_{\ell^n}^a$.
We write $\{a\}_n$ for the unique integer satisfying
\[
	0 \le \{a\}_n < \ell^n/2, \qquad \{a\}_n \equiv \pm a \pmod{\ell^n}.
\]
\begin{lem}\label{lem:quo2}
Let $a_1,\dotsc,a_r \in \Z_{\ell}$ 
and $c_1,\dotsc,c_r \in \Z$.
Suppose
	\begin{enumerate}[(i)]
	\item $c_1 \ne 0$.
	\item $a_1 \not \equiv 0 \pmod{\ell}$.
	\item $a_1 \ne \pm a_2,\pm a_3,\cdots \pm a_r \pmod{\ell^n}$.
	\end{enumerate}
Write
\begin{equation}\label{eqn:varepsilonn}
	\varepsilon_n \; =\; 
	\prod_{1 \le i \le r} (1-\zeta_{\ell^n}^{a_i})^{c_i}.
\end{equation}
Then, $\varepsilon_n \notin \langle \pm \zeta_{\ell^n}, V_{n-1}\rangle$
for all sufficiently large $n$.
\end{lem}
\begin{proof}
If $a_j \equiv 0 \pmod{\ell}$ then $(1-\zeta_{\ell^n}^{a_j}) \in V_{n-1}$.
We may therefore suppose $a_j \not\equiv 0 \pmod{\ell}$
for all $j$.
Write 
\[
		\delta_n \; =\; 
		\prod_{1 \le i \le r} \left(1-\zeta_{\ell^n}^{\{a_i\}_n} \right)^{c_i}.
\]
In view of the identity \eqref{eqn:residueexp} it will
be sufficient to show that $\delta_n \notin \langle \pm \zeta_{\ell^n},
V_{n-1}\rangle$ for $n$ sufficiently large. Also, in view
	of Lemma~\ref{lem:quo}, it is sufficient to show
	for sufficiently large $n$ that
	$\{a_1\}_n \not\equiv \{a_j\}_n \pmod{\ell^n}$
	for all $2 \le j \le n$. This is equivalent
	to $a_1 \ne \pm a_j$ for $2 \le j \le n$
	which is hypothesis (iii). This completes
	the proof.
\end{proof}
The following corollary easily follows from Lemma~\ref{lem:quo2}.
\begin{cor}\label{cor:quo3}
Let $a_1,\dotsc,a_r \in \Z_{\ell}$ 
and $c_1,\dotsc,c_r \in \Z$.
Suppose
\begin{enumerate}[(i)]
\item $c_1 \equiv 1 \pmod{2}$.
\item $a_1 \not \equiv 0 \pmod{\ell}$.
\item $a_1 \ne \pm a_2,\pm a_3,\cdots \pm a_r \pmod{\ell^n}$.
\end{enumerate}
Let $\varepsilon_n$ be as in \eqref{eqn:varepsilonn}.
Then, $\varepsilon_n \notin \langle \pm \zeta_{\ell^n}, V_{n-1}, V_n^2\rangle$
for all sufficiently large $n$.
\end{cor}

\subsection*{Units and $S$-units from cyclotomic polynomials}
For $m \ge 1$, let $\Phi_m(X) \in \Z[X]$ be the 
\textbf{$m$-th cyclotomic polynomial} defined by
\[
	\Phi_m(X) = \prod_{\substack{1 \le i \le m\\(i,m)=1}} (X -\zeta_m^i).
\]
These satisfy the identity \cite[Chapter 2]{Washington}
\begin{equation}\label{eqn:Phim1}
	X^m-1=\prod_{d \mid m} \Phi_d(X).
\end{equation}
It follows from the M\"{o}bius inversion formula that
\begin{equation}\label{eqn:Phim2}
	\Phi_m(X)=\prod_{d \mid m} (X^d-1)^{\mu(m/d)}
\end{equation}
where $\mu$ denotes the M\"{o}bius function.
\begin{lem}\label{lem:cycpol}
Let $\ell$ be a prime and $n \ge 1$. 
Let $m\ge 1$,
and suppose $\ell^n \nmid m$.
\begin{enumerate}[(a)]
	\item $\Phi_m(\zeta_{\ell^n}) \in V_n \subseteq \OO(\Omega_{n,\ell},S)^\times$,
where $S=\{\upsilon_\ell\}$.
\item If $m \ne \ell^u$ for all $u \ge 0$,
	then $\Phi_m(\zeta_{\ell^n}) \in C_n \subseteq \OO(\Omega_{n,\ell})^\times$.
\end{enumerate}
Moreover,
\[
\upsilon_\ell(\Phi_{\ell^t}(\zeta_{\ell^n})) \; = \; \begin{cases}
\frac{1}{\ell^{n-1} (\ell-1)} & t=0\\
\frac{1}{\ell^{n-t}} & 1 \le t \le n-1.\\
\end{cases}
\]
\end{lem}
\begin{proof}
Let $t=\ord_\ell(m)<n$.
Observe that $\Phi_m(X) \mid (X^m-1)$. 
	Hence $\Phi_m(\zeta_{\ell^n}) \cdot  \OO(\Omega_{n,\ell})$ 
	divides $(1-\zeta_{\ell^n}^m) \cdot \OO(\Omega_{n,\ell})$. 
By Lemma~\ref{lem:cycval1} we have
	$(1-\zeta_{\ell^n}^m) \cdot \OO(\Omega_{n,\ell}) \, = \, \lambda_n^{\ell^t}$, giving (a).

For (b), write $m=\ell^t k$ where $k>1$. Then $\Phi_m(X)$ divides
the polynomial $(X^m-1)/(X^{\ell^t}-1)$.
	Therefore $\Phi_m(\zeta) \cdot \OO(\Omega_{n,\ell})$ divides
\[
\frac{(1-\zeta_{\ell^n}^m)}{
	(1-\zeta_{\ell^n}^{\ell^t})} \cdot \OO(\Omega_{n,\ell})=\frac{\lambda_n^{\ell^t}}{\lambda_n^{\ell^t}}=1 \cdot \OO(\Omega_{n,\ell}).
\]
Thus $\Phi_m(\zeta)$ is a unit, giving (b).

The final part of the Lemma follows from Lemma~\ref{lem:cycval1}, and the formulae
\[
\Phi_{\ell^t}(X)=\begin{cases}
X-1 & t=0\\
(X^{\ell^t}-1)/(X^{\ell^{t-1}}-1) & t \ge 1.
\end{cases}
\]
\end{proof}

\begin{lem}
Let $n \ge 2$ if $\ell \ne 2$ and $n \ge 3$ if $\ell=2$.
Then $V_n/\langle \pm \zeta_{\ell^n}\rangle$ is free with basis
\[
	\{\Phi_m(\zeta_{\ell^n}) \; : \; 1 \le m < \ell^n/2, \; \ell \nmid m \}.
\]
\end{lem}
\begin{proof}
This follows from Lemma~\ref{lem:Vnbasis} thanks to 
	identities \eqref{eqn:Phim1} and \eqref{eqn:Phim2}.
\end{proof}

\section{The $S$-unit equation over $\Q(\zeta_{\ell^n})^+$}\label{sec:sunit+}
We continue with the notation 
of the previous section. In particular,
let $K$ be a subfield of $\overline{\Q}$ and $S$
be a finite set of primes of $K$.
Let $k$ be a non-zero rational integer. We shall make frequent use
of the
correspondence between elements of $(\PP^1\setminus \{0,k,\infty\})(\OO(K,S))$
and the set of solutions to the $S$-unit equation
\[
	\varepsilon+\delta=k, \qquad \varepsilon,~\delta \in \OO(K,S)^\times,
\]
sending $\varepsilon \in (\PP^1\setminus \{0,k,\infty\})(\OO(K,S))$
to $(\varepsilon,\delta)=(\varepsilon,k-\varepsilon)$.

Now, as before let $\ell$ be a rational prime, $n$ is a positive integer.
If $\ell=2$ suppose $n \ge 2$.
Write $\zeta=\zeta_{\ell^n}$.
We write $\Omega_{n,\ell}^+=\Q(\zeta+1/\zeta)$ for the index $2$
totally real subfield of $\Omega_{n,\ell}$.
We write
\[
	\Omega_{\infty,\ell}^+=\bigcup_{n=1}^\infty \Omega_{n,\ell}^+.
\]
In this section, for suitable $S$, we produce solutions to the $S$-unit
equations over $\Omega_{\infty,\ell}^+$.

As before, $\Phi_m$ denotes the $m$-cyclotomic polynomial.
It is convenient to record the first few $\Phi_m$:
\begin{gather*}
\Phi_1= X - 1, \qquad
\Phi_2= X + 1, \qquad
\Phi_3= X^2 + X + 1,\\
\Phi_4= X^2 + 1, \qquad
\Phi_5= X^4 + X^3 + X^2 + X + 1,\\
\Phi_6= X^2 - X + 1, \qquad
\Phi_7= X^6 + X^5 + X^4 + X^3 + X^2 + X + 1, \\
\Phi_8= X^4 + 1, \qquad
\Phi_9= X^6 + X^3 + 1, \qquad
\Phi_{10}= X^4 - X^3 + X^2 - X + 1. 
\end{gather*}

We shall call a polynomial $F \in \Z[X]$ \textbf{super-cyclotomic} if it is 
of the form $X^m f_1 f_2 \cdots f_k$ where each $f_i(X)$ is a
cyclotomic polynomial. We know, thanks to Lemma~\ref{lem:cycpol}, that if $F$ is super-cyclotomic and $\ell$ is a prime,
then $F(\zeta_{\ell^n}) \in \OO(\Omega_n,\{\upsilon_\ell\})^\times$ for $n$ sufficiently large.
We wrote a short computer program that lists all super-cyclotomic polynomials
of degree at most $20$ and searches for ternary relations of the form $F - G = k H$
with $F$, $G$, $H$ super-cyclotomic, $\gcd(F,G,H)=1$ and $k$ is a positive integer.
Note that any such relation $F - G=kH$ gives points $\varepsilon_n=F(\zeta_{\ell^n})/H(\zeta_{\ell^n}) \in (\PP^1 \setminus \{0,k,\infty\})(\OO(\Omega_n,\{\upsilon_\ell\}))$, for $n$ sufficiently large.
We found the following ternary relations between super-cyclotomic polynomials.
\begin{gather}
\label{eqn:identity1}	\Phi_2(X)^2-\Phi_3(X) \; =\; X;\\
	\Phi_2(X)^2-\Phi_4(X) \; = \; 2X;\\
	\Phi_2(X)^2-\Phi_6(X) \; =\; 3X;\\
	\Phi_2(X)^2-\Phi_1(X)^2 \; =\; 4X;\\
	\Phi_2(X)^4-\Phi_{10}(X) \;=\; 5X \Phi_3(X);\\
	\Phi_2^2(X) \Phi_3(X)-\Phi_1(X)^2 \Phi_6(X) \; =\;  6 X \Phi_4(X);\\
	\Phi_7(X) - \Phi_1(X)^6 \;=\; 7 X \Phi_6(X)^2; \\
	\Phi_2(X)^4-\Phi_1(X)^4 \; = \;  8 X\Phi_4(X);\\
\label{eqn:identity10}	\Phi_2(X)^4\Phi_5(X)-\Phi_1(X)^4\Phi_{10}(X) \; = \; 10 X\Phi_4(X)^3.
\end{gather}
From the identities \eqref{eqn:Phim1} and \eqref{eqn:Phim2} one easily sees that $F(X^k)$
is super-cyclotomic for any super-cyclotomic polynomial $F$ and any positive integer $k$, thus
each of the nine identities above in fact yields an infinite family of identities.
We pose the following open problems:
\begin{itemize}
\item Are there ternary linear relations between super-cyclotomic polynomials
that are outside these nine families?
\item Classify all ternary linear relations between super-cyclotomic
polynomials.
\end{itemize}
\begin{lem}\label{lem:conj}
Let $c : \Omega_\ell \rightarrow \Omega_\ell$ denote complex conjugation.
Let $n \ge 1$ and let $\zeta=\zeta_{\ell^n}$ be an $\ell^{n}$-th root of $1$.
Let $m \ge 1$ and suppose $\ell^n \nmid m$. Then 
\[
\frac{\Phi_m(\zeta)^c}{\Phi_m(\zeta)} \; = \; \begin{cases}
\zeta^{-\varphi(m)} & m \ge 2\\
-\zeta^{-1} & m=1.
\end{cases}
\]
\end{lem} 
\begin{proof}
Note that $\zeta^c=\zeta^{-1}$. So 
\[
\frac{\Phi_1(\zeta)^c}{\Phi_1(\zeta)}=\frac{\zeta^{-1}-1}{\zeta-1}=-\zeta^{-1}, \qquad
\frac{\Phi_2(\zeta)^c}{\Phi_2(\zeta)}=\frac{\zeta^{-1}+1}{\zeta+1}=\zeta^{-1}. 
\]
Let $m \ge 3$. The polynomial $\Phi_m$ is monic of degree $\varphi(m)$,
and its roots are the primitive $m$-th roots of $1$ which come in distinct
pairs $\eta$, $\eta^{-1}$. Thus the trailing coefficient is $1$.
It follows that $X^{\varphi(m)} \Phi_m(X^{-1})$ is monic and has the same roots as $\Phi_m$,
therefore 
\[
\Phi_m(X)=X^{\varphi(m)} \Phi_m(X^{-1}).
\]
Hence
\[
\frac{\Phi_m(\zeta)^c}{\Phi_m(\zeta)}
=
\frac{\Phi_m(\zeta^{-1})}{\Phi_m(\zeta)}
=\zeta^{-\varphi(m)}.
\]
\end{proof}
\begin{lem}\label{lem:cycsunit}
Let $S=\{\upsilon_\ell\}$. 
Let 
\[
k \in \{1,2,3,4,5,6,7,8,10\}.
\]
	Then $(\PP^1\setminus \{0,k,\infty\})(\OO(\Omega_{\infty,\ell}^+,S))$ is infinite.
\end{lem}
\begin{proof}
The proof makes use of identities \eqref{eqn:identity1}--\eqref{eqn:identity10}.
We prove the lemma for $k=10$ using 
the identity \eqref{eqn:identity10}; the other cases are similar.
Let $n \ge 3$ and let $\zeta=\zeta_{\ell^n}$.
Let
\[
\varepsilon=\frac{\Phi_2(\zeta)^4 \Phi_5(\zeta)}{\zeta \Phi_4(\zeta)^3},
\qquad
\delta=\frac{-\Phi_1(\zeta)^4\Phi_{10}(\zeta)}{\zeta \Phi_4(\zeta)^3}.
\]
By the identity, $\varepsilon+\delta=10$. By Lemma~\ref{lem:cycpol}
we know that $\varepsilon$, $\delta$ are $S$-units.
	A priori, $\varepsilon$, $\delta$ belong to $\Omega_{\infty,\ell}$.
However, an easy application of Lemma~\ref{lem:conj}
shows that $\varepsilon^c=\varepsilon$ and $\delta^c=\delta$,
	so $\varepsilon$, $\delta \in \Omega_{\infty,\ell}^+$.
	It follows that $\varepsilon$ is an $\OO(\Omega_{\infty,\ell}^+,S)$-point 
on $\PP^1 \setminus \{0,10,\infty\}$.
This point depends on $\zeta=\zeta_{\ell^n}$.
Let us make sure that we really obtain infinitely
many such points as we vary $n$. Write
\[
	\varepsilon_n \; =\; \frac{\Phi_2(\zeta_{\ell^n})^4 \Phi_5(\zeta_{\ell^n})}{\zeta_{\ell^n} \Phi_4(\zeta_{\ell^n})^3} \; =\;
	\frac{(1-\zeta_{\ell^n}^2)^7(1-\zeta_{\ell^n}^5)}{\zeta_{\ell^n} (1-\zeta_{\ell^n})^5 (1-\zeta_{\ell^n}^4)^3}
	\in V_n.
\] 
To show that we obtain infinitely many distinct $\varepsilon_n$
	it is enough to show that $\varepsilon_n \notin V_{n-1}$
	for $n$ sufficiently large. This follows by an easy application
	of Lemma~\ref{lem:quo}; to illustrate this let $\ell=5$
	and suppose $\varepsilon_n \in V_{n-1}$. 
	Note that $1-\zeta_{5^n}^5 \in V_{n-1}$. 
	It follows that
	\[
		(1-\zeta_{5^n})^{-5} (1-\zeta_{5^n}^2)^7 (1-\zeta_{5^n}^4)^{-3}
		\; \in \;
		\langle \pm \zeta_{\ell^n} , V_{n-1} \rangle.
	\]
	Now in the product on the left the exponent of $1-\zeta_{5^n}$ is
	$-5$ whereas the exponent of $1-\zeta_{5^n}^{1+5^{n-1}}$ is $0$,
	contradicting Lemma~\ref{lem:quo}. The proof is 
	similar for $\ell=2$, and for $\ell \ne 2$, $5$.
	It follows that 
	we have infinitely many $\OO(\Omega_{\infty,\ell}^+,S)$-points on $\PP^1 \setminus \{0,10,\infty\}$.
\end{proof}

\subsection*{Proof of Theorem~\ref{thm:Siegel2} for $\ell=2$ and $3$}
For $\ell=2$, $3$, we have $\Omega_{\infty,\ell}^+=\Q_{\infty,\ell}$.
Indeed, if $\ell=2$ then $\Q_{n,2}=\Omega_{n+2,2}^+$  and 
if $\ell=3$ then $\Q_{n,3}=\Omega_{n+1,3}^+$. 
Therefore Theorem~\ref{thm:Siegel2} with $\ell=2$ and $3$
follows immediately from Lemma~\ref{lem:cycsunit} for $k \in
\{1,2,3,4,5,6,7,8,10\}$.  

Also, if $\baseprime = 2$, then the infinitely many solutions $\eps + \delta =
6$ yields infinitely many solutions for  $2\eps + 2\delta = 12$ and $4\eps +
4\delta = 24$.  And if $\baseprime = 3$, then the infinitely many solutions
$\eps + \delta = 4$ yields infinitely many solutions $3\eps + 3\delta = 12$,
and similarly infinitely many solutions $\eps + \delta = 8$ yields infinitely
many solutions $3\eps + 3\delta = 24$. This proves Theorem~\ref{thm:Siegel2}
for $\ell=2$, $3$ and $k \in \{12, 24\}$. \qed

\subsection*{Proof of Theorem~\ref{thm:Siegel} for $\ell=2$}
Theorem~\ref{thm:Siegel} for $\ell=2$ is simply a special
case of Theorem~\ref{thm:Siegel2}.\qed

\section{The unit equation over $\Q(\zeta_{\ell^n})^+$}\label{sec:unit+}
For roots of unity $\alpha$, $\beta$, we let
\begin{align*}
	E(\alpha,\beta)& =\frac{\alpha^2+\alpha^{-2}}{
		\left( \alpha \beta^{-1}+\alpha^{-1} \beta\right) (\alpha \beta+\alpha^{-1} \beta^{-1})
		}=
	\frac{\Phi_8(\alpha)}{\Phi_4(\alpha \beta) \Phi_4(\alpha/\beta)},\\
	F(\alpha,\beta) & =\frac{\beta^2+\beta^{-2}}{
		\left( \alpha \beta^{-1}+\alpha^{-1} \beta\right) (\alpha \beta+\alpha^{-1} \beta^{-1})
		}=
	\frac{\Phi_8(\beta)}{\Phi_4(\alpha \beta) \Phi_4(\beta/\alpha)}.
\end{align*}
We easily check that
\begin{equation}\label{eqn:EF} 
	E(\alpha,\beta)+F(\alpha,\beta)=1. 
\end{equation}
\begin{lem}\label{lem:+Siegel}
Suppose $\ell$ is odd and $n \ge 1$. 
	Let $\zeta=\zeta_{\ell^n}$.
	Let $i$, $j$ be integers satisfying
$i$, $j$, $i+j$, $i-j \not \equiv 0 \pmod{\ell^n}$.
Then
	$E(\zeta^i,\zeta^j)$, $F(\zeta^i,\zeta^j) \in \OO(\Omega_{n,\ell}^+)^\times$,
	and satisfy the unit equation 
	\begin{equation}\label{eqn:unitKr}
		\varepsilon+\delta=1, \qquad \varepsilon,~\delta \in \OO(\Omega_{n,\ell}^+)^\times.
	\end{equation}
Moreover,
\begin{equation}\label{eqn:valij}
\upsilon_\ell(E(\zeta^i,\zeta^j)-F(\zeta^i,\zeta^j))=\frac{\ell^{\ord_\ell(i+j)}+\ell^{\ord_\ell(i-j)}}{\ell^{n-1}(\ell-1)}
\end{equation}
\end{lem}
\begin{proof}
	It is clear that $E(\zeta^i,\zeta^j)$, $F(\zeta^i,\zeta^j)$
	are fixed  by complex conjugation $\zeta \mapsto \zeta^{-1}$
	and so belong to $\Omega_{n,\ell}^+$.
By Lemma~\ref{lem:cycpol}, $E(\zeta^i,\zeta^j)$
and $F(\zeta^i,\zeta^j)$ are units.
It remains to check~\eqref{eqn:valij}.
We observe
\[
E(\zeta^i,\zeta^j)-F(\zeta^i,\zeta^j)=\frac{(\zeta^{i-j}-\zeta^{j-i})(\zeta^{i+j} -\zeta^{-i-j})}{(\zeta^{i-j}+\zeta^{j-i})(\zeta^{i+j} +\zeta^{-i-j})}=
	\frac{(\zeta^{2(i-j)}-1) (\zeta^{2(i+j)}-1)}
	{\Phi_4(\zeta^{i-j}) \Phi_4(\zeta^{i+j})}.
\]
	The denominator is a unit by Lemma~\ref{lem:cycpol}.
	Now \eqref{eqn:valij} follows from Lemma~\ref{lem:cycval1}.
\end{proof}
\begin{lem}\label{lem:puninf}
Let $\ell$ be an odd prime.
Then $(\PP^1\setminus \{0,1,\infty\})(\OO(\Omega_{\infty,\ell}^+))$ is infinite.
\end{lem}
\begin{proof}
We deduce this from Lemma~\ref{lem:+Siegel}.
Let us take for example $i=2$ and $j=1$.
Let $n\ge 2$ and let
	\[
		\varepsilon_n=E(\zeta_{\ell^n}^2,\zeta_{\ell^n}),
		\qquad
		\delta_n=F(\zeta_{\ell^n}^2,\zeta_{\ell^n}).
	\]
By Lemma~\ref{lem:+Siegel}, $\varepsilon_n$, $\delta_n \in \OO(\Omega_{\infty,\ell}^+)^\times$
and satisfy $\varepsilon_n+\delta_n=1$. Thus $\varepsilon_n
\in (\PP^1\setminus \{0,1,\infty\})(\OO(\Omega_{\infty,\ell}^+))$. Moreover,
\[
	\upsilon_\ell(2 \varepsilon_n-1)=
	\upsilon_\ell(\varepsilon_n-\delta_n)
	=
	\begin{cases}
		\frac{2}{\ell^{n-1} (\ell-1)} & \ell >3\\
		\frac{2}{3^{n-1}} & \ell=3,
	\end{cases}
\]
	by \eqref{eqn:valij}.
Thus $\varepsilon_n \ne \varepsilon_m$ whenever $n \ne m$.
Hence $(\PP^1\setminus \{0,1,\infty\})(\OO(\Omega_{\infty,\ell}^+))$ is infinite.
\end{proof}

\noindent \textbf{Remark.}
Lemma~\ref{lem:puninf} applies only for $\ell$ odd;
for $\ell=2$ it is easy to show that the statement is false.
Indeed,
and let $\eta_n$ be the prime ideal of $\OO(\Omega_{n,2}^+)$ above $2$.
Then $\OO(\Omega_{n,2}^+)/\eta_n \cong \F_2$, and a solution to 
$\varepsilon+\delta=1$ with $\varepsilon$, $\delta \in
\OO(\Omega_{n,2}^\plus)^\times$ reduced modulo $\eta_n$ gives $1+1
\equiv 1 \pmod{2}$ which is impossible.

\subsection*{Proof of Theorem~\ref{thm:Siegel} for $\ell=3$}
	We recall that $\Q_{\infty,3}=\Omega_{\infty,3}^+$.
	Therefore Theorem~\ref{thm:Siegel} for $\ell=3$
	follows immediately from Lemma~\ref{lem:puninf}. \qed

\section{The $S$-unit equation over $\Q_{\infty,5}$} \label{sec:Q5}
The purpose of the is section is to prove Theorems~\ref{thm:Siegel}
and~\ref{thm:Siegel2} for $\ell=5$. These in fact follow immediately
from the following lemma.
\begin{lem}\label{lem:Q5}
	Let $\upsilon_5$ be the unique prime of $\Q_{\infty,5}$ above $5$,
			and write $S=\{\upsilon_5\}$.
			Then
	\begin{enumerate}[(i)]
		\item $(\PP^1\setminus \{0,k,\infty\})(\OO(\Q_{\infty,5},S))$
	is infinite for $k=1$, $4$;
\item $(\PP^1\setminus \{0,2,\infty\})(\OO(\Q_{\infty,5}))$
	is infinite.
	\end{enumerate}
\end{lem}
\begin{proof}
Let $a \in \Z_5^\times$ 
be the element satisfying 
\[
	a^2=-1, \qquad a \equiv 2 \pmod{5};
\]
such an element exists and is unique by Hensel's Lemma.
Let 
	$\sigma : \Omega_{\infty,5} \rightarrow \Omega_{\infty,5}$
	be the field automorphism satisfying
\[
	\sigma(\zeta_{5^n}) \; = \; \zeta_{5^n}^a
\]
for $n \ge 1$.
Note that $\sigma$ is an automorphism of order $4$,
and fixes a subfield of $\Omega_{\infty,5}$ of index $4$. This subfield
is precisely $\Q_{\infty,5}$.

Let
	\[
	\begin{gathered}
F= (x_1 x_2^2 + x_3 x_4^2) (x_1^2 x_4 + x_2 x_3^2),\\
G=(x_1^2 x_2 + x_3^2 x_4) (x_1 x_4^2 + x_2^2 x_3),\\
H=(x_1-x_3)(x_2-x_4)(x_1 x_2-x_3 x_4)(x_1 x_4-x_2 x_3).
	\end{gathered}
\]
Observe 
 $F$, $G$, $H$ are invariant
under the $4$-cycle $(x_1,x_2,x_3,x_4)$.
One can check that $F-G=H$.
Let $n \ge 2$ and write $\zeta=\zeta_{5^n}$.
Let
	\[
		\varepsilon_n=\frac{F(\zeta,\zeta^{a},\zeta^{a^2},\zeta^{a^3})}{H(\zeta,\zeta^{a},\zeta^{a^2},\zeta^{a^3})},
		\qquad
		\delta_n=-\frac{G(\zeta,\zeta^{a},\zeta^{a^2},\zeta^{a^3})}{H(\zeta,\zeta^{a},\zeta^{a^2},\zeta^{a^3})}.
	\]
	From the identity $F-G=H$ we have
	$\varepsilon_n+\delta_n=1$.
	We shall show that $\varepsilon_n$, $\delta_n \in \OO(\Q_{\infty,5},S)^\times$. 

Since $\sigma$ cyclically permutes
	$\zeta,\zeta^a,\zeta^{-1},\zeta^{-a}$
	we conclude that $f(\zeta,\zeta^a,\zeta^{-1},\zeta^{-a}) \in 
	\Q_{\infty,5}$ for $f=F$, $G$, $H$. Thus
	$\varepsilon_n$, $\delta_n \in \Q_{\infty,5}$.
	Moreover, 
	\begin{gather*}
	F \; =\; x_2 x_3^3 x_4^2 \cdot \Phi_2(x_1 x_2^2/x_3 x_4^2) \Phi_2(x_1^2 x_4/x_2 x_3^2),\\
	G\; =\; x_2^2 x_3^3 x_4 \cdot \Phi_2(x_1^2 x_2/x_3^2 x_4) 
		\Phi_2(x_1 x_4^2/x_2^2 x_3),\\
	H\; =\; x_2 x_3^3 x_4^2 \cdot \Phi_1(x_1/x_3) \cdot \Phi_1(x_2/x_4) 
		\cdot \Phi_1(x_1x_2/x_3x_4) \cdot \Phi_1(x_1 x_4/x_2x_3).
	\end{gather*}
Hence
\[
	\begin{split}
		\varepsilon_n &=\frac{\Phi_2(\zeta^{2+4a}) \Phi_2(\zeta^{4-2a})}{\Phi_1(\zeta^2)\Phi_1(\zeta^{2a})\Phi_1(\zeta^{2+2a})\Phi_1(\zeta^{2-2a})}\\
		&=\frac{(1-\zeta^{4+8a})(1-\zeta^{8-4a})}{(1-\zeta^2)(1-\zeta^{2a})(1-\zeta^{2+2a})(1-\zeta^{2-2a})(1-\zeta^{2+4a})(1-\zeta^{4-2a})}.
	\end{split}
\]
and
\[
	\begin{split}
		\delta_n &=\frac{-\zeta^{2a}\Phi_2(\zeta^{4+2a}) \Phi_2(\zeta^{2-4a})}{\Phi_1(\zeta^2)\Phi_1(\zeta^{2a})\Phi_1(\zeta^{2+2a})\Phi_1(\zeta^{2-2a})}\\
		&=\frac{-\zeta^{2a}(1-\zeta^{8+4a})(1-\zeta^{4-8a})}{(1-\zeta^2)(1-\zeta^{2a})(1-\zeta^{2+2a})(1-\zeta^{2-2a})(1-\zeta^{4+2a})(1-\zeta^{2-4a})}.
	\end{split}
\]
We checked, using the fact that $a \equiv 7 \pmod{25}$,
that the exponents of $\zeta$ in the above expressions
for $\varepsilon_n$ and $\delta_n$
all have $5$-adic valuation $0$ or $1$.
It follows from this that
$\varepsilon_n$, $\delta_n \in V_n \subseteq \OO(\Omega_n,S)^\times$ 
for $n \ge 2$.
Hence $\varepsilon_n$, $\delta_n \in \Q_{\infty,5} \cap \OO(\Omega_n,S)^\times=\OO(\Q_{\infty,5},S)^\times$
for $n \ge 2$. To complete the proof of the lemma
for $k=1$
it is enough to show that $\varepsilon_n \ne \varepsilon_m$
for $n>m$, and for this it is enough to show that
$\varepsilon_n \notin \langle \pm \zeta_{5^n},V_{n-1} \rangle$ for $n \ge 2$.
Since $a \equiv 7 \pmod{25}$ we see that
\[
	4+8a \equiv 10, \qquad 8-4a \equiv 5, \qquad
	2+4a \equiv 5, \qquad 4-2a \equiv 15 \pmod{25}.
\]
Thus the factors
\[
	1-\zeta^{4+8a}, \qquad 1-\zeta^{8-4a},
	\qquad 1-\zeta^{2+4a}, \qquad 1-\zeta^{4-2a}
\]
all belong to $V_{n-1}$. Hence it is enough to show that
\begin{equation}\label{eqn:epssimple}
		(1-\zeta^2)(1-\zeta^{2a})(1-\zeta^{2+2a})(1-\zeta^{2-2a})
\end{equation}
does not belong to $\langle \pm \zeta_{5^n},V_{n-1}\rangle$. However, the exponents
$2$, $2a$, $2+2a$, $2-2a$ are respectively $2$, $4$, $1$, $3$
modulo $5$, and hence certainly distinct modulo $5^{n-1}$.
It follows from Lemma~\ref{lem:quo} that the product \eqref{eqn:epssimple}
does not belong to $\langle \pm \zeta_{5^n},V_{n-1}\rangle$ completing the proof for $k=1$.

The proof for $k=2$ is similar, and is based on the identity
$F-G=2H$ where
\[
\begin{gathered}
	F \; = \; (x_1^2 + x_1 x_3 + x_3^2)(x_2^2+x_2 x_4 + x_4^2) \; = \;
	x_3^2 x_4^2 \cdot \Phi_3(x_1/x_3) \cdot \Phi_3(x_2/x_4),\\
	G \; = \; (x_1^2 - x_1 x_3 + x_3^2)(x_2^2-x_2 x_4 + x_4^2) \; = \;
	x_3^2 x_4^2\cdot \Phi_6(x_1/x_3)\cdot \Phi_6(x_2/x_4),\\
	H \; = \; (x_1 x_4+x_2 x_3)(x_1 x_2+x_3 x_4) \; = \;
	x_2 x_3^2 x_4 \cdot \Phi_2(x_1 x_4/x_2 x_3) \cdot \Phi_2(x_1x_2/x_3x_4),
\end{gathered}
\]
and likewise the proof for $k=4$ is based on the identity $F-G=4H$
where
\[
\begin{gathered}
    F \; = \; (x_1 + x_3)^2 (x_2 + x_4)^2 = x_3^2 x_4^2 \cdot \Phi_2(x_1/x_3)^2 \Phi_2(x_2/x_4)^2, \\
    G \; = \; (x_1 - x_3)^2 (x_2 - x_4)^2 = x_3^2 x_4^2 \cdot \Phi_1(x_1/x_3)^2 \Phi_1(x_2/x_4)^2, \\
    H \; = \; (x_1 x_2 + x_3 x_4)(x_1 x_4 + x_2 x_3) = x_2 x_3^2 x_4 \cdot \Phi_2(x_1 x_2 / x_3 x_4) \Phi_2(x_1 x_4 / x_2 x_3).
\end{gathered}
\]
\end{proof}

\noindent \textbf{Remark.} It is appropriate to remark on how the identities in the above proof were found.
Write
\[
	{\Psi}_m(X,Y)=Y^{\varphi(m)} \Phi_m(X/Y)
\]
for the homogenization of the $m$-th cyclotomic polynomial.
Now consider 
\[
	f(x_1,x_2,x_3,x_4)=\Psi_m(u,v)
\]
where $u$, $v$ are monomials in variables $x_1,x_2,x_3,x_4$.
Let $\ell$ be a prime.  We see that evaluating any such $f$ at $(\zeta^\alpha,\zeta^\beta,\zeta^\gamma,\zeta^\delta)$
gives an element of $V_n$ (provided that it does not vanish). We considered products of such $f$
of total degree up to $20$ and picked out ones that are invariant under the $4$-cycle $(x_1,x_2,x_3,x_4)$,
and searched for ternary relations between them. This yielded the identities used in the above proof.

\begin{proof}[Proof of Theorems~\ref{thm:Siegel} and~\ref{thm:Siegel2} for $\ell=5$]
Theorems~\ref{thm:Siegel}
and~\ref{thm:Siegel2} for $\ell=5$ follow immediately
from  Lemma~\ref{lem:Q5}.
\end{proof}

\section{The $S$-unit equation over $\Q_{\infty,7}$} \label{sec:Q7}
\begin{lem}\label{lem:Q7}
	Let $\upsilon_7$ be the unique prime of $\Q_{\infty,7}$ above $7$,
			and write $S=\{\upsilon_7\}$.
			Then
		$(\PP^1\setminus \{0,1,\infty\})(\OO(\Q_{\infty,7},S))$
	is infinite.
\end{lem}
\begin{proof}
In view of the proof of Lemma~\ref{lem:Q5}, it would be natural
to seek polynomials $F$, $G$, $H$ 
in variables $x_1,\dotsc,x_6$ satisfying
the following properties
\begin{itemize}
\item $F\pm G=H$;
\item $F$, $G$, $H$ are invariant under the $6$-cycle
		$(x_1,x_2,\dotsc,x_6)$;
\item each is a product of polynomials
	\[
		f(x_1,x_2,\dotsc,x_6)=\Psi_m(u,v)
	\]
with $u$, $v$ monomials in $x_1,\dotsc,x_6$.
\end{itemize}
Unfortunately, an extensive search has failed to produce
any such triple of polynomials. We therefore need to proceed
a little differently.

Let $a \in \Z_7$ 
be the element satisfying 
\[
	a^2+a+1=0, \qquad a \equiv 2 \pmod{7};
\]
such an element exists and is unique by Hensel's Lemma.
Let 
	$\sigma$, $c : \Omega_{\infty,7} \rightarrow \Omega_{\infty,7}$
	be the field automorphisms satisfying
\[
	\sigma(\zeta_{7^n}) \; = \; \zeta_{7^n}^a, \qquad
	c(\zeta_{7^n}) \; = \; \zeta_{7^n}^{-1}
\]
for $n \ge 1$. Then $\Q_{\infty,7}$ is the field fixed by  
the subgroup of $\Gal(\Omega_{\infty,7}/\Q)$ generated by $\sigma$ and $c$.
We work with polynomials in variables $x_1$, $x_2$, $x_3$. 
Let
\[
	\begin{gathered}
	F \; =\; (x_1 x_2^2+x_3^3)(x_2 x_3^2+x_1^3)(x_3 x_1^2+x_2^3)\\
	G \; = \; (x_1-x_2)(x_2-x_3)(x_3-x_1)(x_1 x_2-x_3^2) (x_2 x_3-x_1^2)(x_3 x_1-x_2^2)\\
	H \; =\; (x_1^2 x_2+x_3^3)(x_2^2 x_3+x_1^3)(x_3^2 x_1+x_2^3).\\
	\end{gathered}
\]
These satisfy the identity $F-G=H$. Moreover, they are invariant
under the $3$-cycle $(x_1,x_2,x_3)$ and all the factors are 
	of the form $\Psi_m(u,v)$ where $m=1$ or $2$,
	and where $u$, $v$ are suitable monomials in $x_1$, $x_2$, $x_3$.
	Evaluating any of $F$, $G$, $H$ at $(\zeta,\zeta^a,\zeta^{a^2})$
	yields an $S$-unit belonging to 
	$\Omega_{n,7}^{\langle \sigma \rangle}$.
Now we let 
	\[
		F^\prime=\frac{F(x_1^2,x_2^2,x_3^2)}{x_1^6 x_2^6 x_3^6},
		\qquad
		G^\prime=\frac{G(x_1^2,x_2^2,x_3^2)}{x_1^6 x_2^6 x_3^6},
		\qquad
		H^\prime=\frac{H(x_1^2,x_2^2,x_3^2)}{x_1^6 x_2^6 x_3^6}.
	\]
Observe that the rational functions $F^\prime$, $G^\prime$, $H^\prime$
satisfy $F^\prime-G^\prime=H^\prime$ and are moreover invariant
under the $3$-cycle $(x_1,x_2,x_3)$. Moreover, $F^\prime$, $G^\prime$,
$H^\prime$ evaluated at $(\zeta,\zeta^a,\zeta^{a^2})$
yield $S$-units belonging to $\Omega_{n,7}^{\langle \sigma \rangle}$.
We need to check that these in fact belong to
$\Q_{n-1,7}=\Omega_{n,7}^{\langle \sigma, c \rangle}$
and so we need to check that these expressions
are invariant under $c$. This follows immediately on observing
that $F^\prime$, $G^\prime$, $H^\prime$ may be rewritten
as 
\[
	\begin{gathered}
		F^\prime \; =\; 
		\left(\frac{x_1 x_2^2}{x_3^3}+\frac{x_3^3}{x_1 x_2^2}\right)
		\left(\frac{x_2 x_3^2}{x_1^3}+\frac{x_1^3}{x_2 x_3^2}\right)
		\left(\frac{x_3 x_1^2}{x_2^3}+\frac{x_2^3}{x_3 x_1^2}\right)\\
	G^\prime \; = \; 
		\left(\frac{x_1}{x_2}-\frac{x_2}{x_1} \right)
		\left(\frac{x_2}{x_3}-\frac{x_3}{x_2} \right)
		\left(\frac{x_3}{x_1}-\frac{x_1}{x_3} \right)
		\left(\frac{x_1 x_2}{x_3^2}-\frac{x_3^2}{x_1 x_2}\right) 
		\left(\frac{x_2 x_3}{x_1^2}-\frac{x_1^2}{x_2 x_3} \right)
		\left(\frac{x_3 x_1}{x_2^2} - \frac{x_2^2}{x_3 x_1} \right)\\
		H^\prime \; =\; 
		\left(\frac{x_1^2 x_2}{x_3^3}+\frac{x_3^3}{x_1^2 x_2}\right)
		\left(\frac{x_2^2 x_3}{x_1^3}+\frac{x_1^3}{x_2^2 x_3}\right)
		\left(\frac{x_3^2 x_1}{x_2^3}+\frac{x_2^3}{x_3^3 x_1}\right).\\
	\end{gathered}
\]
Thus $F^\prime$, $G^\prime$, $H^\prime$ evaluated at $(\zeta,\zeta^a,\zeta^{a^2})$ yield elements of $\OO(\Q_{\infty,7},S)^\times$.
We write
\[
	\varepsilon_n=\frac{F^\prime(\zeta,\zeta^a,\zeta^{a^2})}{H^\prime(\zeta,\zeta^a,\zeta^{a^2})}, \qquad
	\delta_n=-\frac{G^\prime(\zeta,\zeta^a,\zeta^{a^2})}{H^\prime(\zeta,\zeta^a,\zeta^{a^2})}.
\]
Then $\varepsilon_n$, $\delta_n$ belong to $\OO(\Q_{\infty,7},S)^\times$
and satisfy $\varepsilon_n+\delta_n=1$. 
In fact it is straightforward to check that
$\varepsilon_n \notin \langle \pm \zeta_{7^n}, V_{n-1} \rangle$,
from which it follows that $\varepsilon_n \ne \varepsilon_m$
for $n>m$. The details are similar to those of the proof of 
Lemma~\ref{lem:Q5} and we omit them.
\end{proof}

\section{Isogeny classes of elliptic curves over $\Q_{\infty,\ell}$}\label{sec:elliptic}

The purpose of this section is to prove Theorem~\ref{thm:finite}.
Since isogenous elliptic curves share the
same set of bad primes, the corresponding
theorem over number fields is an
immediate consequence of Shafarevich's theorem.
However, as we intend to show in the following
section, Shafarevich's theorem does not
generalize to elliptic curves over
$\Q_{\infty,\ell}$. 
We shall instead rely on a theorem of Kato
to control $\Q_{\infty,\ell}$-points on certain
modular Jacobians.

\bigskip

Our first lemma shows that there are only finitely many
primes that can divide the degree of a
cyclic isogeny of $E$.
\begin{lem}\label{lem:Serre1}
Let $\ell$ be a prime and let $E/\Q_{\infty,\ell}$ be an elliptic curve
without potential complex multiplication.
Then there is a constant $B$, depending on $E$,
such that for primes $p \ge B$, the elliptic curve
$E$ has no $p$-isogenies defined over $\Q_{\infty,\ell}$.
\end{lem}
\begin{proof}
Let $n$ be the least positive integer such
that $E$ admits a model defined over $\Q_{n,\ell}$.
	By a famous theorem of Serre \cite{Serre72}, there is a constant
$B$, depending on $E$, such that for $p \ge B$ 
the mod $p$ representation
\[
	\overline{\rho}_{E,p} \; : \; \Gal(\overline{\Q}/\Q_{n,\ell})
	\rightarrow \GL_2(\F_p)
\]
is surjective. We may suppose that $B \ge 5$. 
Thus, for $p \ge B$, the Galois group
$\Gal(\Q_{n,\ell}(E[p])/\Q_{n,\ell})$ is isomorphic to $\GL_2(\F_p)$
which is non-solvable. 
We will show that $E$ has no $p$-isogeny defined over
$\Q_{\infty,\ell}$. Suppose otherwise. Then
such an isogeny is in fact defined over
$\Q_{m,\ell}$ for some $m \ge n$.
It follows that 
the extension $\Q_{m,\ell}(E[p])/\Q_{m,\ell}$ 
has Galois group
isomorphic to a subgroup of a Borel
subgroup of $\GL_2(\F_p)$, with is solvable.
As the extension $\Q_{m,\ell}/\Q_{n,\ell}$
is cyclic, we conclude that
$\Q_{m,\ell}(E[p])/\Q_{n,\ell}$
is solvable. 
However, this contains the non-solvable
subextension
$\Q_{n,\ell}(E[p])/\Q_{n,\ell}$, giving a contradiction.
\end{proof}

We shall make use of the following theorem of Kato 
\cite[Theorem 14.4]{Kato} building
on work of Rohrlich \cite{Rohrlich}.
\begin{thm}[Kato]
Let $\ell$ be a prime.
Let $A$ be an abelian variety defined over $\Q$
and admitting a surjective map $J_1(N) \rightarrow A$
for some $N \ge 1$. Then $A(\Q_{\infty,\ell})$ 
is finitely generated.
\end{thm}
\begin{lem}\label{lem:Serre2}
Let $p$, $\ell$ be primes. Let $E$ be an elliptic curve
defined over $\Q_{\infty,\ell}$ without
potential complex multiplication. Then, for $m$ sufficiently
large, $E$ has no $p^m$-isogenies defined over $\Q_{\infty,\ell}$.
\end{lem}
\begin{proof}
Let $r$ be the least positive integer such that
the modular curve
$X=X_0(p^r)$ has genus at least $2$,
and write $J=J_0(p^r)$ for the corresponding
modular Jacobian.
It follows from Kato's theorem that 
$J(\Q_{\infty,\ell})$ is finitely generated,
and therefore that $J(\Q_{\infty,\ell})=J(\Q_{n,\ell})$
for some $n \ge 1$. Consider the Abel-Jacobi map
\[
	X \hookrightarrow J, \qquad P \mapsto [P-\infty]
\]
where $\infty \in X(\Q)$ denotes the infinity cusp.
It follows from this embedding that $X(\Q_{\infty,\ell})=X(\Q_{n,\ell})$.
By Faltings' theorem, this set is finite.

Let $k=\# X(\Q_{\infty,\ell})$ and let 
$s=kr$. To prove the lemma we in fact
show that $E$ has no
cyclic isogenies of degree $p^s$ defined
over $\Q_{\infty,\ell}$. Suppose otherwise,
and let $\psi : E \rightarrow E^\prime$ be a
cyclic isogeny of degree $p^s$
defined over $\Q_{\infty,\ell}$. Then, 
we may factor $\psi$ into a sequence 
of cyclic isogenies defined over $\Q_{\infty,\ell}$
\[
	E=E_0 \; \xrightarrow{\psi_1} 
	\; E_1 \; \xrightarrow{\psi_2} \; E_2 \;
	\cdots \xrightarrow{\psi_k} \; E_k=E^\prime
\]
where $\psi_i$ is degree $p^r$.
Note that $E_i$ and $E_j$ are non-isomorphic over $\overline{\Q}$
for $i \ne j$; indeed they are related by a cyclic isogeny
and $E$ does not have potential complex multiplication.
Thus the elliptic curves $E_0,E_1,\dotsc,E_k$
support distinct $\Q_{\infty,\ell}$-points on $X=X_0(p^r)$.
This contradicts the fact that $\# X(\Q_{\infty,\ell})=k$.
\end{proof}

\noindent \textbf{Remark.}
A famous theorem of Serre
\cite[Section 2.1]{SerreAbelian}
asserts that the $p$-adic Tate module
of a non-CM elliptic curve defined over a number field
is irreducible. It is in fact possible to deduce Lemma~\ref{lem:Serre2}
from Serre's theorem for $\ell \ne p$, but we have been
unable to do this for $\ell=p$.

\begin{proof}[Proof of Theorem~\ref{thm:finite}]
Let $E^\prime$ belong to the $\Q_{\infty,\ell}$-isogeny
class of $E$. Let $\psi : E \rightarrow E^\prime$
be an isogeny defined over $\Q_{\infty,\ell}$.
This has kernel of the form $\Z/a \times \Z/ab$
where $a$, $b$ are positive integers, and so it can be
	factored into a composition
\[
	E \rightarrow E/E[a] \cong E \rightarrow E^\prime
\]
where the final morphism is cyclic of degree $b$. Thus 
to prove the proposition, it is enough to show that $E$
has finitely many cyclic
isogenies defined over $\Q_{\infty,\ell}$.
The degree of any such isogeny is divisible
by primes $p < B$ where $B$ is as in Lemma~\ref{lem:Serre1}.
Also, for any $p<B$, we know the exponent of $p$
in the degree of a cyclic isogeny $E \rightarrow E^\prime$
is bounded by Lemma~\ref{lem:Serre2}. 
Thus there are finitely many cyclic isogenies
of $E$ defined over $\Q_{\infty,\ell}$.
\end{proof}

\section{From $S$-unit equations to elliptic curves}\label{sec:antishafarevich}
The aim of this section is to prove Theorem~\ref{thm:Shafarevich}.
We start by recalling a few facts about Legendre elliptic curves
(Proposition III.1.7 of \cite{SilvermanAEC}
and its proof).
Let $K$ be a field of characteristic $\ne 2$ and let $\lambda \in (\PP^1 \setminus \{0,1,\infty\})(K)$.
Associated to $\lambda$ is the Legendre elliptic curve
\[
	E_\lambda \; : \; Y^2 = X(X-1)(X-\lambda).
\]
This model respectively has discriminant and $j$-invariant
\begin{equation}\label{eqn:jinv}
	\Delta=16 \lambda^2(1-\lambda)^2, \qquad
	j=\frac{64 (\lambda^2-\lambda+1)^3}{\lambda^2 (1-\lambda)^2}.
\end{equation}
Moreover, for $\lambda$, $\mu \in (\PP^1 \setminus \{0,1,\infty\})(K)$,
the Legendre elliptic curves $E_\lambda$ and $E_\mu$ are isomorphic over $K$
(or over $\overline{K}$) if and only if
\[
	\mu \; \in \; \left\{ 
	\lambda, \, \frac{1}{\lambda}, \, 1-\lambda, \, 
	\frac{1}{1-\lambda}, \, \frac{\lambda}{\lambda-1}, \,
	\frac{\lambda-1}{\lambda}
	\right\}.
\]
Now let $K$ be a number field and $S$ a finite set of non-archimedean places.
We let $S^\prime$ be the set of non-archimedean places
which are either in $S$ or above $2$. We let
$\lambda \in (\PP^1 \setminus \{0,1,\infty\})(\OO(K,S))$.
Then $\lambda$, $1-\lambda \in \OO(K,S)^\times$.
It follows from the expression
for the discriminant that $E_\lambda$ has good reduction away for $S^\prime$.

\subsection*{Proof of Theorem~\ref{thm:Shafarevich}}
Let $\ell=2$, $3$, $5$ or $7$. 
Let $S$ be given by \eqref{eqn:S}
and let
$S^\prime=S \cup \{\upsilon_2\}$
as in the statement of Theorem~\ref{thm:Shafarevich}.
In proving Theorem~\ref{thm:Siegel}
we constructed, for each positive integer $n$,
elements $\varepsilon_n$, $\delta_n=1-\varepsilon_n$, belonging
$\Q_{\infty,\ell} \cap V_n \subseteq \OO(\Q_{\infty,\ell},S)^\times$,
and moreover verified, for $n \ge 2$, that
$\varepsilon_n \notin \langle \zeta_{\ell^n},V_{n-1} \rangle$.
We let
\[
	E_n \; : \; Y^2=X(X-1)(X-\varepsilon_n).
\]
Then $E_n$ is defined over $\Q_{\infty,\ell}$ and has good reduction
away from $S^\prime$. We claim, for $n>m$,
that $E_n$ and $E_m$ are not isomorphic, even over $\overline{\Q}$.
To see this, suppose $E_n$ and $E_m$ are isomorphic. Then
$\varepsilon_n$ equals one of $\varepsilon_m^{\pm 1}$,
$\delta_m^{\pm 1}$, $(-\varepsilon_m \delta_m)^{\pm 1}$.
This gives a contradiction as all of these belong
to $\langle \pm \zeta_{\ell^n},V_{n-1} \rangle$.
This proves the claim.

It remains to show that the $E_n$ form infinitely many isogeny
classes over $\Q_{\infty,\ell}$. However, this immediately
follows from Theorem~\ref{thm:finite} and the following
lemma. \qed

\begin{lem}
For $n$ sufficiently large, $E_n$ does not have potential
complex multiplication.
\end{lem}
\begin{proof}
Suppose $E_n$ has potential complex multiplication by an order $R$
in an imaginary quadratic field $K$.
Write $j=j(E_n)$.
By standard CM theory \cite[Theorem 5.7]{Shimura}, 
we know that  $\Gal(K(j)/K) \cong \Pic(R)$
and $[\Q(j) : \Q]=[K(j):K]$.
Since in our case $\Q(j)/\Q$ is Galois,
$\Gal(\Q(j)/\Q) \cong \Gal(K(j)/K) \cong \Pic(R)$.
However, $\Q(j) \subset \Q_{\infty,\ell}$ is totally real. 
It follows \cite[page 124]{Shimura}
that $\Pic(R)$ is an elementary abelian $2$-group.
Since $\Q(j) \subset \Q_{\infty,\ell}$,
the Galois group of $\Q(j)/\Q$ is cyclic of order $\ell^n$
	for some $n$.
Thus, $j \in \Q$ if $\ell \ne 2$,
and $j \in \Q_{1,2}=\Q(\sqrt{2})$ if $\ell =2$.
However, from the expression for $j$ in \eqref{eqn:jinv}
we know that $[\Q(\varepsilon_n) : \Q(j)] \le 6$. 
Thus $\varepsilon_n$ belongs to a subfield
of $\Q_{\infty,\ell}$ of degree at most $12$.
The lemma follows since, by Siegel's theorem, the $S$-unit equation
has only finitely many solutions in any number field.
\end{proof}

\section{Hyperelliptic curves over $\Q_{\infty,\ell}$
with few bad primes}\label{sec:hyp}
Let $\ell$ be an odd prime.
Let $g \ge 2$ be an  integer satisfying
\begin{equation}\label{eqn:genus}
\begin{cases}
 \text{$g \equiv (\ell-3)/4\, $ or $\, -1  \pmod{(\ell-1)/2}$}  &  \text{ if $\ell \equiv 3 \pmod{4}$}\\
\text{$g \equiv -1  \pmod{(\ell-1)/4}$} & \text{ if $\ell \equiv 1 \pmod{4}$}. 
\end{cases}
\end{equation}
Then there is a positive integer $k$ such that
\begin{equation}\label{eqn:gk}
k \cdot \left(\frac{\ell-1}{2}\right) \; = \; 
\begin{cases}
 \text{$2g+1$ or $2g+2$} & \text{if $\ell \equiv 3 \pmod{4}$}\\
 2g+2 & \text{if $\ell \equiv 1 \pmod{4}$}.
\end{cases}
\end{equation}
Let $n \ge 2$ be a positive integer satisfying 
\begin{equation}\label{eqn:k}
 \ell^{n-1} \; \ge \; k.
\end{equation}
In this section we construct
a hyperelliptic $D_n$ curve of genus $g$ defined over $\Q_{n-1,\ell}$
with good reduction away from the primes above $2$, $\ell$.

Write
\[
\cZ_n=\{ \zeta \in \Omega_{n,\ell} \quad : \quad \zeta^{\ell^n}=1, \quad \zeta^{\ell^i} \ne 1 \textrm{ if } i < n \} 
\]
for the set of primitive $\ell^n$-th roots of $1$. Write
\[
\cZ_n^+ \; = \; \{ \zeta+\zeta^{-1} \; : \; \zeta \in \cZ_n\} \; \subset \; \Omega_{n,\ell}^+.
\]
We note that any element of $\cZ_n^+$ generates $\Omega_{n,\ell}^+$.
\begin{lem}
$\displaystyle \# \cZ_n^+ = 
\baseprime^{n-1} (\baseprime - 1)/2
$.
\end{lem}
\begin{proof}
We note that $\# \cZ_n =
\varphi(\baseprime^\level) = \baseprime^{\level - 1} (\baseprime - 1)
$. Suppose $\alpha$, $\beta \in \cZ_n$. Then
\begin{equation}\label{eqn:alphabeta}
(\alpha+\alpha^{-1})-(\beta+\beta^{-1}) \; = \; \alpha^{-1} \cdot (1-\alpha \beta) \cdot (1-\alpha \beta^{-1}).
\end{equation}
Thus $\alpha+\alpha^{-1}=\beta+\beta^{-1}$ if and only if $\alpha=\beta$ or $\alpha=\beta^{-1}$.
The lemma follows.
\end{proof}
Write
\[
G_n=\Gal(\Omega_{n,\ell}^+/\Q_{n-1,\ell}), \qquad
H_n=\Gal(\Omega_{n,\ell}^+/\Omega_{n-1,\ell}^+).
\]
We note that these are both cyclic subgroups
of $\Gal(\Omega_{n,\ell}^+/\Q)$ having orders
\[
	\# G_n = (\ell-1)/2, \qquad \#H_n=\ell.
\]
\begin{lem}\label{lem:etadef}
Fix $\zeta \in \cZ_n$.
Let
\begin{equation}\label{eqn:etadef}
	\eta_i \; =  \; \zeta^{1+\ell^{n-1}(i-1)} + \zeta^{-1-\ell^{n-1}(i-1)},
	\qquad 1 \le i \le \ell.
\end{equation}
Then $\eta_1,\dotsc,\eta_\ell \in \cZ_n^+$
form a single orbit under the action of $H_n$,
but have pairwise disjoint orbits under the action of $G_n$.
\end{lem}
\begin{proof}
Let $\kappa \in \Gal(\Omega_{n,\ell}/\Q)$
be given by $\kappa(\zeta)=\zeta^{1+\ell^{n-1}}$.
We note that $\kappa$ has order $\ell$ and fixes
$\Omega_{n-1,\ell}$. We denote
the restriction of $\kappa$ to $\Omega_{n,\ell}^+$
by $\tau$; this is a cyclic generator of $H_n$.
Note that
\[
\eta_i=\tau^{i-1} (\zeta+\zeta^{-1}), \qquad 1 \le i \le \ell.
\]
Let $\sigma_1$, $\sigma_2 \in G_n$. 
Let $1 \le i < j \le \ell$ and suppose 
$\sigma_1 (\eta_i) =\sigma_2 (\eta_j)$.
Thus $\sigma_1 \tau^{i-1} (\eta_1)=\sigma_2 \tau^{j-1} (\eta_1)$,
so $\tau^{1-j} \sigma_2^{-1} \sigma_1 \tau^{i-1}$ fixes $\eta_1$.
As $\eta_1$ generates $\Omega_{n,\ell}^+$, we have $\tau^{1-j} \sigma_2^{-1} \sigma_1 \tau^{i-1}=1$
is the identity element in $\Gal(\Omega_{n,\ell}^+/\Q)$. However, 
$\Gal(\Omega_{n,\ell}^+/\Q)$
is abelian, so
\[
\tau^{i-j}=\sigma_1^{-1} \sigma_2 \in G_n \cap H_n=\{1\}.
\]
Since $1 \le i \le j \le \ell$ and $\tau$ has order $\ell$ we have $i=j$.
\end{proof}
The Galois group $G_n$ 
acts faithfully on $\cZ_n^+$. This action has $\ell^{n-1}$ orbits.
Assumption \eqref{eqn:k} ensures that the number of orbits is at least $k$. 
If $k>\ell$, then we \textbf{extend} the list $\eta_1,\dotsc,\eta_\ell \in \cZ_n^+$
to $\eta_1,\dotsc,\eta_k \in \cZ_n^+$, so that the $\eta_i$ continue to have
disjoint orbits under the action of $G_n$; if $\ell=3$ the choice
of $\eta_4$ will be important later, and we choose 
$\eta_4=\zeta^2+\zeta^{-2}$.
Consider the curve 
\begin{equation}\label{eqn:Dn}
D_n \; : \; Y^2=\prod_{j=1}^{k} \prod_{\sigma \in G_n} (X-\eta_j^\sigma).
\end{equation}
\begin{lem} \label{lem:hyperelliptic}
The curve $D_n$ is hyperelliptic of genus $g$,
is defined over $\Q_{n-1,\ell}$,
and has good reduction away from the primes above $2$
and $\ell$.
\end{lem}
\begin{proof}
Our assumption on the orbits ensures that the polynomial on the right hand-side
	of \eqref{eqn:Dn}
is separable. By \eqref{eqn:gk}, the degree of the polynomial is either $2g+1$ or $2g+2$.
Thus $D_n$ is a hyperelliptic curve of genus $g$.
A priori, $D_n$ is defined over $\Omega_{n,\ell}^+$. However, 
the roots of the hyperelliptic
polynomial are permuted by the action 
	of $G_n=\Gal(\Omega_{n,\ell}^+/\Q_{n-1,\ell})$
and so the polynomial belongs to $\Q_{n-1,\ell}[X]$.
Hence $D_n$ is defined over $\Q_{n-1,\ell}$.

Let $u_1,\dotsc,u_{d}$ be the roots of the hyperelliptic polynomial.
Then the discriminant of hyperelliptic polynomial is 
\[
\prod_{1 \le i < j \le d} (u_i-u_j)^2.
\]
However, $u_i$, $u_j$ are distinct elements of $\cZ_n^+$. Thus 
there are $\alpha$, $\beta \in \cZ_n$
with $\alpha \ne \beta$, $\beta^{-1}$ such that $u_i=\alpha+\alpha^{-1}$, $u_j=\beta+\beta^{-1}$.
	From the identity \eqref{eqn:alphabeta},
\[
u_i-u_j 
	\; = \;
	\alpha^{-1}(1-\alpha \beta^{-1})(1-\alpha \beta).
\]
Since $\alpha\beta$ and $\alpha \beta^{-1}$ are non-trivial $\ell$-power roots of $1$,
we see that $u_i-u_j$ is a $\{\upsilon_\ell\}$-unit, and hence  the discriminant
of the hyperelliptic polynomial of $D_n$ is a $\{\upsilon_\ell\}$-unit. 
\end{proof}

Given four pairwise 
distinct elements $z_1$, $z_2$, $z_3$, $z_4$ of a field $K$, we shall 
employ the notation $(z_1,z_2 \, ;\, z_3,z_4)$ to denote the
\textbf{cross ratio}
\[
	(z_1,z_2 \, ;\, z_3,z_4) \; = \; \frac{(z_1-z_3)(z_2-z_4)}{(z_1-z_4)(z_2-z_3)}.
\]
We extend the cross ratio to four distinct elements $z_1,z_2,z_3,z_4$ 
of $\PP^1(K)$ in the usual way. We let $\GL_2(K)$ act on $\PP^1(K)$
via fractional linear transformations
\[
	\gamma(z)=\frac{az+b}{cz+d}, \qquad 
	\gamma=\begin{pmatrix} a & b\\ c & d\end{pmatrix}.
\]
It is well-known and easy to check that these fractional
linear transformations leave the cross ratio unchanged:
\[
	(\gamma(z_1),\gamma(z_2) \, ;\, \gamma(z_3),\gamma(z_4))
	\; = \; (z_1,z_2 \, ;\, z_3,z_4).
\]
\begin{lem}\label{lem:cross}
Let $\overline{K}$ be an algebraically closed field
of characteristic $0$. Let
\[
		D \; : \; Y^2=\prod_{i=1}^{d} (X-a_i), 
		\qquad D^\prime \; : \; Y^2=\prod_{i=1}^{d} (X-b_i),
\]
be genus $g$ curves defined over $\overline{K}$ where the polynomials
on the right are separable. If $D$, $D^\prime$
are isomorphic then there is some permutation
$\mu \in S_{d}$ such that for all 
quadruples of pairwise distinct indices
$1 \le r,s,t,u \le d$
\[
		(a_{r} , a_{s} \, ;\,  a_{t}, a_{u} ) \;
		= \;
		(b_{\mu(r)},b_{\mu(s)}\, ;\, b_{\mu(t)}, b_{\mu(u)}).
\]
\end{lem}
\begin{proof}
We shall make use of the following standard
description
(e.g.\ \cite[Proposition 6.11]{Bakergenus2}) 
of isomorphisms of hyperelliptic curves:
every isomorphism $\pi \; : \; D \rightarrow D^\prime$
is of the form
\[
		\pi(X,Y) \; = \;
		\left(\frac{aX+b}{cX+d} , \frac{eY}{(cX+d)^{g+1}} \right)
\]
for some
\[
		\gamma=\begin{pmatrix} a & b\\ c & d\end{pmatrix} \in \GL_2(\overline{K}),
		\qquad e \in \overline{K}^\times.
\]
Observe that $\pi(a_i,0)$ has $Y$-coordinate $0$;
thus
\[
		\{ \gamma(a_1),\dotsc,\gamma(a_{d})\}
		\; = \; \{ b_1,\dotsc,b_{d}\}.
\]
Hence there is a permutation $\mu \in S_{d}$ such that
$\gamma(a_i)=b_{\mu(i)}$. The lemma follows
from the invariance of the cross ratio under
the action of $\GL_2(\overline{K})$.
\end{proof}

\begin{lem}\label{lem:count}
Let $\ell \ge 11$ be prime. Then there is some $a \in \Z_{\ell}^\times$
of order $\ell-1$ such that
        \begin{multline}\label{eqn:notequiv}
        1+a^2 \; \not\equiv \;   0, \; \pm (1-a^2),\, \pm (a+a^3),\,
        \pm (a-a^3),\\
         \pm (1+a^3),\, \pm (1-a^3),\, \pm (a+a^2),\,
        \pm (a-a^2)  \pmod{\ell}.
\end{multline}
\end{lem}
\begin{proof}
Making use of
the fact that a polynomial of degree $n$ has
at most $n$ roots, we see that the number
of $a \in \F_\ell$ that \textbf{do not satisfy} \eqref{eqn:notequiv}
is (very crudely) bounded by $37$.
An element $a \in \Z_{\ell}^\times$ of order $\ell-1$
is the unique Hensel lift of an element
$a \in \F_\ell^\times$
of order $\ell-1$. There are precisely $\varphi(\ell-1)$
elements of order $\ell-1$ in $\F_\ell^\times$.
A theorem of Shapiro \cite[page 23]{Shapiro},
asserts that $\varphi(n) > n^{\log{2}/\log{3}}$ for $n \ge 30$.
We note that if $\ell \ge 317$ then
$\varphi(\ell-1) \ge 316^{\log{2}/\log{3}} \approx 37.8$,
and so the lemma holds for $\ell  \ge 317$.
For the range $11 \le \ell \le 317$ we checked
the lemma by brute force computer enumeration.
\end{proof}

\begin{lem}\label{lem:noniso}
Let $n>m$ be sufficiently large.
Then $D_n$ and $D_m$ are non-isomorphic, even over $\overline{\Q}$.
\end{lem}
\begin{proof}
Note that all roots of the hyperelliptic polynomial for $D_n$
in \eqref{eqn:Dn} belong to $\cZ_n^+$. It follows
from \eqref{eqn:alphabeta} that the cross ratio of
any four of them belongs to $V_n$.
Suppose $D_n$ and $D_m$ are isomorphic.
Let $u_1,u_2,u_3,u_4$ be any distinct roots of the hyperelliptic
polynomial for $D_n$ given in \eqref{eqn:Dn}. Then,
by Lemma~\ref{lem:cross},
\[
        (u_1,u_2 \, ; \, u_3, u_4) \; \in V_m \subseteq V_{n-1}.
\]
We shall obtain a contradiction through a careful
choice of the four roots $u_1,\dotsc,u_4$.

We first suppose that $k \ge 2$ and $\ell \ge 5$.
Let $\zeta=\zeta_{\ell^n}$
and $b=1+\ell^{n-1}$. Then, by Lemma~\ref{lem:etadef},
$\eta_1=\zeta+\zeta^{-1}$ and $\eta_2=\zeta^b+\zeta^{-b}$.
Let $a \in \Z_{\ell}^\times$ have order $\ell-1$.
Let $\kappa \in \Gal(\Omega_{n,\ell}/\Q_{n-1,\ell})$
be given by $\kappa(\zeta)=\zeta^a$.
Then $\kappa$ is a cyclic generator
for $\Gal(\Omega_{n,\ell}/\Q_{n-1,\ell})$.
We shall denote the restriction of $\kappa$ to
$\Omega_{n,\ell}^+$ by $\mu$. Then $\mu$ is a cyclic generator
for $G_n=\Gal(\Omega_{n,\ell}^+/\Q_{n-1,\ell})$
having order $(\ell-1)/2$. We shall take
\[
    \begin{aligned}
        u_1 &= \eta_1 = \zeta + \zeta^{-1}, & \quad u_2 &= \mu(\eta_1) = \zeta^a + \zeta^{-a}, \\
        \quad u_3 &= \eta_2 = \zeta^b + \zeta^{-b}, & \quad u_4 &= \mu(\eta_2) = \zeta^{ab} + \zeta^{-ab}.
    \end{aligned}
\]
We compute the cross ratio
with the help of identity \eqref{eqn:alphabeta},
finding
\begin{equation*}
        (u_1, u_2 \, ; \, u_3, u_4) \; = \; 
\frac{(1-\zeta^{1+b})(1-\zeta^{1-b})(1-\zeta^{a+ab})(1-\zeta^{a-ab})}{(1-\zeta^{1+ab})(1-\zeta^{1-ab})(1-\zeta^{a+b})(1-\zeta^{a-b})}.
\end{equation*}
As $b \equiv 1 \pmod{\baseprime}$, and clearly $a \not \equiv \pm 1
\pmod{\baseprime}$, it is easy to check that $1+b$ is the only one out
of the eight exponents of $\zeta$ above that is 
$\pm 2 \pmod{\baseprime}$.
Therefore by Lemma \ref{lem:quo2}, the cross ratio is not an element of
$\langle \pm \zeta_{\baseprime^n}, V_{n-1} \rangle$ for $n$
sufficiently large, giving a contradiction
for the case $k \ge 2$ and $\ell \ge 5$.

Next we suppose that $k=1$. It follows from \eqref{eqn:gk}
that $\ell \ge 11$.
We choose $a \in \Z_\ell^\times$ as in Lemma~\ref{lem:count},
and, as above, take $\mu$ to be the corresponding
generator of $G_n$ of order $(\ell-1)/2 \ge 5$.
We take
\[
        u_i=\mu^{i-1}(\eta_1)=\zeta^{a^{i-1}}+\zeta^{-a^{i-1}},
        \qquad 1 \le i \le 4;
\]
observe that these are four roots of the hyperelliptic polynomial
of $D_n$ given in \eqref{eqn:Dn}.
The assumption that $\ell \ge 11$ ensures
that $a$ has order $\ge 10$ and so
$u_1,u_2,u_3,u_4$ are indeed pairwise distinct.
We compute the cross ratio
with the help of identity \eqref{eqn:alphabeta},
finding
\[
        (u_1,u_2 \, ; \, u_3,u_4) \; = \;
        \frac{(1-\zeta^{1+a^2})(1-\zeta^{1-a^2})(1-\zeta^{a+a^3})(1-\zeta^{a-a^3})}{(1-\zeta^{1+a^3})(1-\zeta^{1-a^3})(1-\zeta^{a+a^2})(1-\zeta^{a-a^2})}.
\]
Using Lemma~\ref{lem:quo} and our choice of $a$ given by Lemma~\ref{lem:count}
we conclude that this cross ratio does not belong
to $\langle \pm \zeta_{\ell^n},V_{n-1}\rangle$
for $n$ sufficiently large. This gives 
a contradiction for the case $k=1$.

Finally, we consider $\baseprime = 3$. 
It follows from \eqref{eqn:gk} that $k \ge 5$. 
Recall our choices of $\eta_1$, $\eta_2$, $\eta_3$
in Lemma~\ref{lem:etadef}, and our choice of $\eta_4=\zeta^2+\zeta^{-2}$
in the particular case $\ell=3$.
We choose the four roots $u_i=\eta_i$ for $i=1,\dotsc,4$,
and obtain,
\[
        (u_1,u_2 \, ; \, u_3,u_4) \; = \;
	\frac{(1-\zeta^{2+2 \times 3^{n-1}})(1-\zeta^{-2 \times 3^{n-1}})(1-\zeta^{3+3^{n-1}})(1-\zeta^{-1+3^{n-1}})}{(1-\zeta^3)(1-\zeta^{-1})(1-\zeta^2)(1-\zeta^{-3^{n-1}})}.
\]
As before, with the help of Lemma~\ref{lem:quo2},
we easily verify that the cross ratio is not an element of 
$\langle \pm \zeta_{\baseprime^n}, V_{n-1} \rangle$ for $n$ sufficiently large. This completes the proof. 
\end{proof}

\subsection*{Proof of Theorem~\ref{thm:Shafarevich2}}  
If $\baseprime = 3$ or $5$ then
\eqref{eqn:genus} does not impose any restriction on the genus.
 Therefore we obtain, as above,
for every genus $g \ge 2$, infinitely many $\overline{\Q}$-isomorphism
classes of genus $g$ hyperelliptic curves, defined
over $\Q_{\infty,\ell}$, with good reduction
away from $\{\upsilon_2,\upsilon_\ell\}$.

It remains to deal with $\ell=7$, $11$ and $13$.
Here, \eqref{eqn:genus} imposes the restriction
\[
	g \equiv \begin{cases}
		\text{$1$ or $2 \bmod{3}$} & \text{if $\ell=7$}\\
		\text{$2$ or $4 \bmod{5}$} & \text{if $\ell=11$}\\
		\text{$2 \bmod{3}$} & \text{if $\ell=13$}.
	\end{cases}
\]
We very briefly sketch how to remove the restriction.
Instead of $D_n$ defined as in \eqref{eqn:Dn},
we consider the more general
\[
	D_n \; : \; Y^2=h(X) \cdot \prod_{j=1}^{k} \prod_{\sigma \in G_n} (X-\eta_j^\sigma)
\]
where 
\begin{itemize}
\item $h$ is a monic divisor of $X(X-1)(X+1)$; 
\item $k$ and $h$
are chosen to obtain the desired genus; 
\item $\eta_j \in \cZ_n^+$ are chosen as before. 
\end{itemize}
These
$D_n$ are clearly defined over $\Q_{n-1,\ell}$.
To check that they have good reduction away from
$S^\prime=\{\upsilon_2,\upsilon_\ell\}$,  we need
to verify that
the difference of any two distinct roots $u$, $v$
of the hyperelliptic polynomial belongs to
$\OO(\Omega_n,S^\prime)^\times$. 
The proof of Lemma~\ref{lem:hyperelliptic}
shows this if $u$, $v \in \cZ_n^+$.
For the remaining possible differences 
it is enough to note that
\[
	\alpha+\alpha^{-1}=\alpha^{-1} \Phi_4(\alpha), \qquad	\alpha+\alpha^{-1} +1 = \alpha^{-1} \Phi_3(\alpha),
	\qquad 
	\alpha+\alpha^{-1}-1=\alpha^{-1} \Phi_6(\alpha)
\]
which are all units by Lemma~\ref{lem:cycpol}.
We omit the remaining details.
\qed

\section{Isogeny classes of hyperelliptic curves over $\Q_{\infty,\ell}$}\label{sec:isog}
A beautiful theorem of Kummer asserts that the index of the cyclotomic units
$C_n$ in the full unit group $\OO(\Omega_{n,\ell})^\times$ equals the class number $h_n^+$ of $\Omega_{n,\ell}^+$.
In this section, with the help of Kummer's theorem, we prove for certain primes $\ell$ the existence of infinitely many
isogeny classes of hyperelliptic Jacobians over $\Q_{\infty,\ell}$ with good reduction away from $\ell$.
We first prove a few elementary lemmas.
\begin{lem} \label{lem:fieldsquareroots}
 Let $K$ be a field of characteristic not $2$, and let $L = K(\sqrt{\alpha_1},
\dots, \sqrt{\alpha_r})$ where $\alpha_i \in K^\times$. Then for any $x
\in K$ such that $\sqrt{x} \in L$, we have 
\begin{equation*}
    x \; =\; \alpha_1^{e_1} \cdots \alpha_r^{e_r} q^2
\end{equation*}
for some integers $e_i \in \mathbb{Z}$ and $q \in K$.
\end{lem}
\begin{proof}
Let $M$ be a field of characteristic not $2$, and let $d \in M$ be a non-square. Let $x \in M$ and suppose $\sqrt{x} \in M(\sqrt{d})$.
Then $\sqrt{x} = y + z \sqrt{d}$ for some $y$, $z \in M$. Squaring, we deduce that $yz=0$. 
Thus $x=y^2$ or $x=dz^2$.

We now prove the lemma by induction on $r$.  The above establishes the case $r = 1$.  
Let $r \geq 2$, and let $x \in K$ satisfy $\sqrt{x} \in L$.  
Letting $M=K(\sqrt{\alpha_1},\dotsc,\sqrt{\alpha_{r-1}})$
we see that $x \in M$ and $\sqrt{x} \in M(\sqrt{\alpha_r})$.
Thus, by the above, $\sqrt{x} \in M$ or $\sqrt{x \alpha_r} \in M$.
In other words, 
\[
	\sqrt{x \cdot \alpha_r^e} \; \in\;  M=K(\sqrt{\alpha_1},\dotsc,\sqrt{\alpha_{r-1}})
\]
for some $e \in \{0,1\}$.
By the inductive hypothesis, there are $e_1,\dotsc,e_{r-1} \in \Z$ and $q \in K$ such that
\[
		x \cdot \alpha_r^e \; = \; \alpha_1^{e_1} \cdots \alpha_{r-1}^{e_{r-1}} q^2.
\]
The proof is complete on taking $e_r=-e$.
\end{proof}

\begin{lem} \label{lem:classnumber}
Let $\baseprime$ be an odd prime. Let $q \in \Omega_{\infty, \baseprime}$ satisfy
$q^2 \in V_n$. If the class number $h_n^+$ of $\Omega_{\level, \baseprime}^+$ is odd, then $q \in V_n$.
\end{lem}
\begin{proof}
Let $q \in \Omega_{\infty,\ell}$
satisfy $q^2 \in V_n \subset \Omega_{n,\ell}$.
As the extension $\Omega_{\infty,\ell}/\Omega_{n,\ell}$
is pro-$\ell$, we conclude that $q \in \Omega_{n,\ell}$.
However, $V_n \subseteq \OO(\Omega_{n,\ell},\{\upsilon_\ell\})^\times$,
where, as usual, $\upsilon_\ell$ denotes the prime above $\ell$.
Thus $q \in \OO(\Omega_{n,\ell},\{\upsilon_\ell\})^\times$.
We claim that
\[
	[\OO(\Omega_{n,\ell},\{\upsilon_\ell\})^\times : V_n] \; = \; h_n^\plus.
\]
The lemma follows immediately from the claim.
To prove the claim, consider the commutative diagram with exact rows
\[
	\begin{tikzcd}
		1 \arrow{r}   & C_n \arrow{r}{} \arrow[hookrightarrow]{d} & V_n \arrow{r}{\kappa} \arrow[hookrightarrow]{d} & \Z \arrow{r}{}
		\arrow{d} & 1\\
		1 \arrow{r}{} & \OO(\Omega_{n,\ell})^\times \arrow{r}{} & \OO(\Omega_{n,\ell},\{\upsilon_\ell\})^\times \arrow{r}{\kappa} & \Z \arrow{r}{} & 1
	\end{tikzcd}
\]
where $\kappa(\alpha)=\ord_{(1-\zeta)}(\alpha)$.
By the snake lemma, 
\[
	\OO(\Omega_{n,\ell},\{\upsilon_\ell\})^\times/V_n \; \cong
	\;
	\OO(\Omega_{n,\ell})^\times / C_n.
\]
Write $C_n^+ = C_n \cap \Omega_{n,\ell}^+$.
The aforementioned theorem of Kummer asserts that
\[
	[\OO(\Omega_{n,\ell})^\times : C_n] \; = \;
	[\OO(\Omega_{n,\ell}^\plus)^\times : C_n^\plus] \; = \; h_n^\plus;
\]
see, for example, \cite[Exercise 8.5]{Washington}
for the first equality, and \cite[Theorem 8.2]{Washington}
for the second. This proves the claim.
\end{proof}

\begin{lem} \label{lem:4torsion}
Let $K$ be a field of characteristic $\ne 2$. Let $f \in K[X]$ be a monic
separable polynomial of odd degree $d \ge 5$. Write $f= \prod_{i=1}^d (X-\alpha_i)$
with $\alpha_i \in \overline{K}$.
Let $C/K$ be a hyperelliptic curve given by $Y^2 = f(X)$ with Jacobian $J$. Then
\[
	K(J[2])=K(\alpha_1,\dotsc,\alpha_d),
	\qquad
	K(J[4])=K(J[2])\left(\big\{ \sqrt{\alpha_i - \alpha_j} \big\}_{1 \leq i,j \leq d} \right).
\]
\end{lem}

\begin{proof}
Write $\infty$ for the point at infinity on the given model for $C$.
The expression given for $K(J[2])$ is well-known; it may be seen by observing 
(see, for example \cite{Schaefer2descent}) that the 
	classes of the classes of degree $0$ divisors $[(\alpha_i,0)-\infty]$ with $i=1,\dotsc,d$ generate $J[2]$.

Yelton \cite[Theorem 1.2.2]{Yelton} gives a high-powered proof of the given expression
for $K(J[4])$. For the convenience
of the reader we give a more elementary argument. Let $L=K(J[2])$. The theory
of $2$-descent on hyperelliptic Jacobians furnishes, for any field $M \supseteq L$, an
	injective homomorphism
	\cite{Schaefer2descent}, \cite{Stoll2descent}
\[
	J(M)/2J(M) \; \hookrightarrow \; \prod_{i=1}^{d} M^*/(M^*)^2
\]
known as the $X-\Theta$-map. This in particular sends the $2$-torsion point $[(\alpha_i,0)-\infty]$
to 
\[
	\left((\alpha_i-\alpha_1) \, , \dotsc \, ,\, (\alpha_{i}-\alpha_{i-1}) \, ,\,
	\prod_{j \ne i} (\alpha_i-\alpha_j) \, , \,
	(\alpha_{i}-\alpha_{i+1}) \, , \, \dotsc \, , \, (\alpha_i-\alpha_d)\right).
\]
	The field $K(J[4])$ is the smallest extension of $M$ of $L$
	such that all the images of the $2$-torsion generators $[(\alpha_i,0)-\infty]$
	are trivial in $\prod_{i=1}^d M^*/(M^*)^2$. This is plainly the extension
\[
 	M=L
	\left(\big\{ \sqrt{\alpha_i - \alpha_j} \big\}_{1 \leq i,j \leq d} \right).
\]
\end{proof}

\begin{lem}\label{lem:prim}
	Let $p$ be a prime for which $2$ is a primitive root (i.e.\ $2$ is a generator for $\F_p^\times$). 
Let $G$ be a cyclic group of order $p$,
	and let $V$ be an $\F_2[G]$-module with $\dim_{\F_2}(V)=p-1$. Suppose that the
action of $G$ on $V\setminus \{0\}$ is free. Then $V$ is irreducible.
\end{lem}
\begin{proof}
Let $W$ be a $\F_2[G]$-submodule of $V$, and write $d=\dim_{\F_2}(W)$.
Since the action of $G$ on $V\setminus \{0\}$ is free,
	the set $W\setminus \{0\}$ consists of $G$-orbits,
	all having size $p$. However,
	$\# (W \setminus \{0\})=2^{d}-1$,
and so $p \mid (2^d-1)$. By assumption, $2$ is a primitive root modulo $p$,
therefore $(p-1) \mid d$. Since $W$ is an $\F_2$-subspace of $V$ which has dimension $p-1$,
we see that $W=0$ or $W=V$.
\end{proof}

\begin{lem} \label{lem:evenisog}
Let $\ell=2p+1$, where $\ell$ and $p$ are odd primes. Suppose $2$ is a primitive root modulo $p$.
Let $g=(\ell-3)/4$. Let $n \ge 2$ and let
$D_n/\Q_{n-1,\ell}$ be the hyperelliptic curve defined in Section~\ref{sec:hyp}.
Let $A/\Q_{\infty,\ell}$ be an abelian variety and let
$\phi : J(D_n) \rightarrow A$ be an isogeny defined over $\Q_{\infty,\ell}$.
Then $\phi=2^r \phi_{\mathrm{odd}}$ where $\phi_{\mathrm{odd}} : J(D_n) \rightarrow A$
is an isogeny of odd degree.
\end{lem}
We remark if $\ell$ and $p$ are primes with $\ell =2p+1$ then
$p$ is called a Sophie-Germain prime, and $\ell$ is called
as safe prime.
\begin{proof}[Proof of Lemma~\ref{lem:evenisog}]
Note that, in the notation of Section~\ref{sec:hyp},
$k=1$, and the hyperelliptic polynomial for $D_n$ has odd degree $2g+1=(\ell-1)/2=p$,
and consists of a single orbit under action of $G_n=\Gal(\Omega_n^+/\Q_{n-1,\ell})$:
\[
		D_n \; : \; y^2 \; = \; \prod_{\sigma \in G_n} (X-\eta_1^\sigma),
	\qquad \eta_1=\zeta_{\ell^n}+\zeta_{\ell^n}^{-1}.
\]
In particular, the hyperelliptic polynomial is irreducible
over $\Q_{\infty,\ell}$. 
It follows from this (e.g. \cite[Lemma 4.3]{Stoll2descent})
that $J(\Q_{\infty,\ell})[2]=0$,
where $J$ denotes $J(D_n)$ for convenience.
We note, by Lemma~\ref{lem:4torsion},
that $\Q_{\infty,\ell}(J[2])=\Q_{\infty,\ell}(\eta_1)=\Omega_{\infty,\ell}^+$.
We consider the action of $G_{\infty}:=\Gal(\Omega_{\infty,\ell}^+/\Q_{\infty,\ell})$
on $J[2]$. The group $G_\infty$
is cyclic of order $(\ell-1)/2=p$. Any element fixed by this action
belongs to $J(\Q_{\infty,\ell})[2]=0$. Thus $G_{\infty}$ acts freely on $V\setminus \{0\}$,
where $V:=J[2]$.

Now let $\phi : J \rightarrow A$ be an isogeny defined over $\Q_{\infty,\ell}$.
Then $W:=\ker(\phi) \cap J[2]$ is a subgroup of $V$ stable under the
action of $G_\infty$, and therefore an $\F_2[G_\infty]$-submodule of the $\F_2[G_\infty]$-module $V$.
Observe that $\dim_{\F_2}(V)=2g=p-1$.
By hypothesis, $2$ is a primitive root modulo $p$.
We apply Lemma~\ref{lem:prim} to deduce that $W=0$ or $W=V$.
Therefore, either $\phi$ already has odd degree,
or $J[2] \subseteq \ker(\phi)$.
In the latter case, observe that
$\phi=2 \phi^\prime$ where $\phi^\prime : J \rightarrow A$
is an isogeny defined over $\Q_{\infty,\ell}$ of degree $\deg(\phi)/2^{2g}$. As $\phi$ has finite degree,
by repeating the argument we eventually arrive at $\phi=2^r \phi_{\mathrm{odd}}$.
\end{proof}


\begin{lem} \label{lem:oddisog}
Let $\ell=2p+1$, where $\ell$ and $p$ are odd primes. Suppose $2$ is a primitive root modulo $p$.
Suppose that the class number $h_n^+$ of $\Omega_{n,\ell}^+$ is odd for all $n$.
Let $g=(\ell-3)/4$. For $n \ge 2$ let
$D_n/\Q_{n-1,\ell}$ be the genus $g$ hyperelliptic curve defined in Section~\ref{sec:hyp}.
Let $n > m$ be sufficiently large.
Then there are no isogenies $J(D_n) \to J(D_m)$ defined over $\Q_{\infty,\ell}$.
\end{lem}
The assumption that $h_n^+$ is odd for all $n$ may seem at first sight very
restrictive. However, it is conjectured \cite{BPR} that $h_{n+1}^+=h_n^+$ for all
but finitely many pairs $(\ell,n)$. Moreover, Washington \cite{Washington2} has shown
that $\ord_p(h_n)$ remains bounded as $n \rightarrow \infty$, for any fixed prime $p$.
\begin{proof}[Proof of Lemma~\ref{lem:oddisog}]
Write $J_n$ for $J(D_n)$.
Suppose there is an isogeny $\phi : J_n \rightarrow J_m$ defined over $\Q_{\infty,\ell}$.
By Lemma~\ref{lem:evenisog} we may suppose that $\phi$ has odd degree,
and so $\ker(\phi) \cap J_n[4]=0$.
Thus $\phi$ restricted to $J_n[4]$ induces an isomorphism
of $\Gal(\overline{\Q}/\Q_{\infty,\ell})$-modules
$J_n[4] \cong J_m[4]$. In particular,
$\Q_{\infty,\ell}(J_n[4])=\Q_{\infty,\ell}(J_m[4])$.
As in the proof of Lemma~\ref{lem:evenisog} we have
$\Q_{\infty,\ell}(J_n[2])=\Q_{\infty,\ell}(J_m[2])=\Omega_{\infty,\ell}^+$.
Thus, by Lemma \ref{lem:4torsion}, the
equality 
$\Q_{\infty,\ell}(J_n[4])=\Q_{\infty,\ell}(J_m[4])$
may be rewritten as 
\begin{equation*}
    \Omega_{\infty,\baseprime}^+ \Big(\big\{ \sqrt{\realgen_{n, i} - \realgen_{n, j} } \big\}_{1 \leq i, j \leq (\baseprime-1)/2 } \Big) 
	\; =\;  \Omega_{\infty,\baseprime}^+ \Big(\big\{ \sqrt{\realgen_{m, i} - \realgen_{m, j} } \big\}_{1 \leq i, j \leq (\baseprime-1)/2} \Big)
\end{equation*}
where $\realgen_{r, i} := \mu_r^{i-1} (\zeta_{\baseprime^r} + \zeta_{\baseprime^r}^{-1})$ where $\mu_r$ is a cyclic generator of $G_r$.
This, in particular, implies that
\begin{equation*}
    \sqrt{\realgen_{n, 2} - \realgen_{n, 1}} \; \in \; \Omega_{\infty,\baseprime}^+ \Big(\big\{ \sqrt{\realgen_{m, i} - \realgen_{m, j} } \big\}_{1 \leq i, j \leq (\baseprime-1)/2} \Big)
\end{equation*}
We apply Lemma \ref{lem:fieldsquareroots} to obtain
\begin{align*}
    \realgen_{n, 2} - \realgen_{n, 1}
    =
    \pm
    &\prod_{1 \leq i< j \leq \frac{\baseprime - 1}{2}}
    (\realgen_{m, i} - \realgen_{m, j})
    ^{e_{i,j}}  \cdot q^2
\end{align*}
for some integers $e_{i,j} \in \mathbb{Z}$ and $q \in \Omega_{\infty, \baseprime}^+$.   
By Lemma \ref{lem:classnumber}, we have $q \in V_n$.  
The generator $\mu_n$ of $G_n$ is given by $\mu_n(\zeta_{\ell^n}+\zeta_{\ell^n}^{-1})=\zeta_{\ell^n}^a+\zeta_{\ell^n}^{-a}$
where $a \in \Z_\ell^\times$ has order $(\ell-1)$.
Note
\begin{equation*}
    \realgen_{n, 2} - \realgen_{n, 1} \; = \;
	\zeta_{\baseprime^n}^a + \zeta_{\baseprime^n}^{-a} - \zeta_{\baseprime^n} - \zeta_{\baseprime^n}^{-1} 
	\; = \; \zeta_{\baseprime^n}^{-a} (1 - \zeta_{\baseprime^n}^{a+1})(1 - \zeta_{\baseprime^n}^{a-1}).
\end{equation*}
Thus,
\begin{equation*}
    (1 - \zeta_{\baseprime^n}^{a+1}) (1 - \zeta_{\baseprime^n}^{a-1} )  \; \in \; \langle \pm \zeta_{\baseprime^n}, V_m, V_n^2 \rangle .
\end{equation*}
However, $(a+1) \not \equiv \pm (a-1) \pmod{\ell}$.
Now Corollary~\ref{cor:quo3} gives a contradiction.
\end{proof}


\subsection*{Proof of Theorem~\ref{thm:Shafarevich3}}
Let $\ell \ge 11$. Let
\[
	g \; =\; \lfloor (\ell-3)/4 \rfloor \; = \;
	\begin{cases}
		(\ell-3)/4 & \text{$\ell \equiv 3 \pmod{4}$}\\
		(\ell-5)/4 & \text{$\ell \equiv 1 \pmod{4}$}.
	\end{cases}
\]
Thus $g$ satisfies \eqref{eqn:genus}. Let $D_n$ be as in Section~\ref{sec:hyp}.
By Lemma~\ref{lem:hyperelliptic}, the hyperelliptic curve $D_n/\Q_{n-1,\ell}$ has genus $g$,
and good reduction away from $\{\upsilon_2,\upsilon_\ell\}$. Moreover, by Lemma~\ref{lem:noniso},
we have $D_n$ and $D_m$ are non-isomorphic, even over $\overline{\Q}$, for $n>m$ sufficiently
large.

Now suppose
\begin{enumerate}[(i)]
\item $\ell=2p+1$ where $p$ is also an odd prime;
\item $2$ as a primitive root modulo $p$.
\end{enumerate}
It then follows from Lemma~\ref{lem:oddisog} that $J(D_n)$
and $J(D_m)$ are non-isogenous over $\Q_{\infty,\ell}$ provided
$h_n^+$ is odd for all $n$, where $h_n^+$ denotes the 
class number of $\Omega_{n,\ell}^+$. Write $h_n$ for the class number
of $\Omega_{n,\ell}$. 
It is known thanks to the work of Estes \cite{Estes} that $h_1$
is odd for all primes $\ell$ satisfying (i) and (ii) (a simplified
proof of this result is given Stevenhagen \cite[Corollary 2.3]{Stevenhagen}).
Moreover,
Ichimura and Nakajima \cite{IchimuraNakajima} show,
for primes $\ell \le 509$, that the ratio
$h_n/h_1$ is odd for all $n$.
The primes $11 \le \ell \le 509$ satisfying both (i) and (ii) are
$11$, $23$, $59$, $107$, $167$, $263$, $347$, $359$.
Thus for these primes $h_n$ is odd for all $n$.
As $h_n^+ \mid h_n$ (see for example \cite[Theorem 4.10]{Washington}), we know for these primes that $h_n^+$ is odd
for all $n$.
This completes the proof. \qed

\bigskip

\noindent \textbf{Remark.}
\begin{itemize}
\item A key step in our proof of Theorem~\ref{thm:Shafarevich3}
is showing that $J(D_n)[2]$ is irreducible as an $\F_2[G_\infty]$-module
whenever $\ell=2p+1$ where $p$ is a prime having $2$ as a primitive
root. It can be shown for all other $\ell$ that the $\F_2[G_\infty]$-module
$J(D_n)[2]$ is in fact reducible. 
\item Another key step is the argument in the proof of Lemma~\ref{lem:oddisog}
showing that for $n>m$ sufficiently large,
the Jacobians $J(D_n)$ and $J(D_m)$ are not related via odd degree
		isogenies defined over $\Q_{\infty,\ell}$.
This step
can be made to work, with very minor modifications
to the argument, for all $\ell \ge 11$, and all 
choices of genus $g$ given in \eqref{eqn:genus}.
\end{itemize}

\bibliographystyle{abbrv}
\bibliography{samir}

\begin{thebibliography}{10}

\bibitem{Abramovich}
D.~Abramovich.
\newblock Birational geometry for number theorists.
\newblock In {\em Arithmetic geometry}, volume~8 of {\em Clay Math. Proc.},
  pages 335--373. Amer. Math. Soc., Providence, RI, 2009.

\bibitem{Bakergenus2}
M.~H. Baker, E.~Gonz\'{a}lez-Jim\'{e}nez, J.~Gonz\'{a}lez, and B.~Poonen.
\newblock Finiteness results for modular curves of genus at least $2$.
\newblock {\em Amer. J. Math.}, 127(6):1325--1387, 2005.

\bibitem{BPR}
J.~Buhler, C.~Pomerance, and L.~Robertson.
\newblock Heuristics for class numbers of prime-power real cyclotomic fields.
\newblock In {\em High primes and misdemeanours: lectures in honour of the 60th
  birthday of {H}ugh {C}owie {W}illiams}, volume~41 of {\em Fields Inst.
  Commun.}, pages 149--157. Amer. Math. Soc., Providence, RI, 2004.

\bibitem{Estes}
D.~R. Estes.
\newblock On the parity of the class number of the field of {$q$}th roots of
  unity.
\newblock volume~19, pages 675--682. 1989.
\newblock Quadratic forms and real algebraic geometry (Corvallis, OR, 1986).

\bibitem{Faltings}
G.~Faltings.
\newblock Endlichkeitss\"{a}tze f\"{u}r abelsche {V}ariet\"{a}ten \"{u}ber
  {Z}ahlk\"{o}rpern.
\newblock {\em Invent. Math.}, 73(3):349--366, 1983.

\bibitem{FKS2}
N.~Freitas, A.~Kraus, and S.~Siksek.
\newblock On asymptotic {F}ermat over {$\Bbb{Z}_p$}-extensions of {$\Bbb{Q}$}.
\newblock {\em Algebra Number Theory}, 14(9):2571--2574, 2020.

\bibitem{GreenbergParkCity}
R.~Greenberg.
\newblock Introduction to {I}wasawa theory for elliptic curves.
\newblock In {\em Arithmetic algebraic geometry ({P}ark {C}ity, {UT}, 1999)},
  volume~9 of {\em IAS/Park City Math. Ser.}, pages 407--464. Amer. Math. Soc.,
  Providence, RI, 2001.

\bibitem{IchimuraNakajima}
H.~Ichimura and S.~Nakajima.
\newblock On the 2-part of the class numbers of cyclotomic fields of prime
  power conductors.
\newblock {\em J. Math. Soc. Japan}, 64(1):317--342, 2012.

\bibitem{Kato}
K.~Kato.
\newblock {$p$}-adic {H}odge theory and values of zeta functions of modular
  forms.
\newblock Number 295, pages ix, 117--290. 2004.
\newblock Cohomologies $p$-adiques et applications arithm\'{e}tiques. III.

\bibitem{Mazur_1972}
B.~Mazur.
\newblock Rational points of abelian varieties with values in towers of number
  fields.
\newblock {\em Invent. Math.}, 18:183--266, 1972.

\bibitem{Rohrlich}
D.~E. Rohrlich.
\newblock On {$L$}-functions of elliptic curves and cyclotomic towers.
\newblock {\em Invent. Math.}, 75(3):409--423, 1984.

\bibitem{Schaefer2descent}
E.~F. Schaefer.
\newblock {$2$}-descent on the {J}acobians of hyperelliptic curves.
\newblock {\em J. Number Theory}, 51(2):219--232, 1995.

\bibitem{SerreAbelian}
J.-P. Serre.
\newblock {\em Abelian {$l$}-adic representations and elliptic curves}.
\newblock W. A. Benjamin, Inc., New York-Amsterdam, 1968.
\newblock McGill University lecture notes written with the collaboration of
  Willem Kuyk and John Labute.

\bibitem{Serre72}
J.-P. Serre.
\newblock Propri\'{e}t\'{e}s galoisiennes des points d'ordre fini des courbes
  elliptiques.
\newblock {\em Invent. Math.}, 15(4):259--331, 1972.

\bibitem{Shapiro}
H.~Shapiro.
\newblock An arithmetic function arising from the {$\phi$} function.
\newblock {\em Amer. Math. Monthly}, 50:18--30, 1943.

\bibitem{Shimura}
G.~Shimura.
\newblock {\em Introduction to the arithmetic theory of automorphic functions}.
\newblock Kan\^{o} Memorial Lectures, No. 1. Iwanami Shoten Publishers, Tokyo;
  Princeton University Press, Princeton, N.J., 1971.
\newblock Publications of the Mathematical Society of Japan, No. 11.

\bibitem{SilvermanAEC}
J.~H. Silverman.
\newblock {\em The arithmetic of elliptic curves}, volume 106 of {\em Graduate
  Texts in Mathematics}.
\newblock Springer-Verlag, New York, 1986.

\bibitem{Stevenhagen}
P.~Stevenhagen.
\newblock Class number parity for the {$p$}th cyclotomic field.
\newblock {\em Math. Comp.}, 63(208):773--784, 1994.

\bibitem{Stoll2descent}
M.~Stoll.
\newblock Implementing 2-descent for {J}acobians of hyperelliptic curves.
\newblock {\em Acta Arith.}, 98(3):245--277, 2001.

\bibitem{Shafarevich}
I.~R. \v{S}afarevi\v{c}.
\newblock Algebraic number fields.
\newblock In {\em Proc. {I}nternat. {C}ongr. {M}athematicians ({S}tockholm,
  1962)}, pages 163--176. Inst. Mittag-Leffler, Djursholm, 1963.

\bibitem{Washington2}
L.~C. Washington.
\newblock The non-{$p$}-part of the class number in a cyclotomic {${\bf
  Z}_{p}$}-extension.
\newblock {\em Invent. Math.}, 49(1):87--97, 1978.

\bibitem{Washington}
L.~C. Washington.
\newblock {\em Introduction to cyclotomic fields}, volume~83 of {\em Graduate
  Texts in Mathematics}.
\newblock Springer-Verlag, New York, second edition, 1997.

\bibitem{Yelton}
J.~S. Yelton.
\newblock {\em Hyperelliptic {J}acobians and their associated l-adic {G}alois
  representations}.
\newblock ProQuest LLC, Ann Arbor, MI, 2015.
\newblock Thesis (Ph.D.)--The Pennsylvania State University.

\bibitem{Zarhin_Parshin}
Y.~G. Zarhin and A.~N. Parshin.
\newblock Finiteness problems in {D}iophantine geometry, 2009.

\bibitem{Zarhin_2010}
Y.~G. Zarkhin.
\newblock Endomorphisms of abelian varieties, cyclotomic extensions, and {L}ie
  algebras.
\newblock {\em Mat. Sb.}, 201(12):93--102, 2010.

\end{thebibliography}

\end{document}